\documentclass[11pt]{amsart}

\usepackage{amsmath}
\usepackage{amssymb}
\usepackage{mathrsfs}
\usepackage[all]{xy}

\usepackage{graphicx}
\usepackage{latexsym}
\usepackage{amsxtra}
\usepackage{url}
\usepackage{verbatim}

\usepackage{tikz}
\usetikzlibrary{shapes.misc, positioning}

\usepackage{float}

\usepackage[urlcolor=blue]{hyperref}

\newtheorem{theorem}{Theorem}[section]
\newtheorem{claim}[theorem]{Claim}

\newtheorem{lemma}[theorem]{Lemma}

\newtheorem{proposition}[theorem]{Proposition}
\newtheorem{corollary}[theorem]{Corollary}

\theoremstyle{definition}
\newtheorem{definition}[theorem]{Definition}

\newtheorem{question}[theorem]{Question}

\theoremstyle{remark}
\newtheorem{remark}[theorem]{Remark}

\newtheorem{fact}[theorem]{Fact}

\def\l{{\langle}}
\def\r{{\rangle}}

\newcount\skewfactor
\def\mathunderaccent#1#2 {\let\theaccent#1\skewfactor#2
\mathpalette\putaccentunder}
\def\putaccentunder#1#2{\oalign{$#1#2$\crcr\hidewidth
\vbox to.2ex{\hbox{$#1\skew\skewfactor\theaccent{}$}\vss}\hidewidth}}

\newcommand{\lusim}[1]{\smash{\underset{\raisebox{1.2pt}[0cm][0cm]{$\sim$}}
{{#1}}}}
\def\smallbox#1{\leavevmode\thinspace\hbox{\vrule\vtop{\vbox
   {\hrule\kern1pt\hbox{\vphantom{\tt/}\thinspace{\tt#1}\thinspace}}
   \kern1pt\hrule}\vrule}\thinspace}

\DeclareMathOperator{\dom}{dom}

\title{On ultrapowers and cohesive ultrafilters}
\author{Tom Benhamou}
\thanks{This research  was supported by the National Science Foundation under Grant
No. DMS-2346680}
\address[Benhamou]{Department of Mathematics, Rutgers University, New Brunswick ,NJ USA}
\email{tom.benhamou@rutgers.edu}


\begin{document}
\begin{abstract}
    We characterize the Tukey order, the Galvin property/ Cohesive ultrafilters from \cite{Kanamori1978} in terms of ultrapowers. We use this characterization to measure the distance between the Tukey order and other well-known orders of ultrafilters. Secondly, we improve two theorems of Kanamori \cite{Kanamori1978} from the 70's. We then study the point spectrum and the depth spectrum of an ultrafilter, and give a simple positive answer to Kanamori's question \cite[Question 2]{Kanamori1978} starting from a supercompact cardinal. We also prove that a positive answer requires more than $o(\kappa)=\kappa^{++}$. Finally, we prove several consistency results regarding the point and depth spectrum.
\end{abstract}

\maketitle
\section{Introduction}
Among the most elegant applications of set theory, involve ultrafilters \cite{KomjathBook,HindimanScott}. For example, in Topology, ultrafilters can be used to define the Stone-\v{C}ech compactification, provide examples of topological spaces with special properties, and in  Moore-Smith convergence of nets. The latter also motivates the study of the Tukey order \cite{Tukey} which studies cofinal types of partially ordered sets. The Tukey order was studied extensively, both on general directed sets and on sets of the form $(\mathcal{X},\preceq)$, $(\mathcal{X},\preceq^*)$ where $\mathcal{X}$ is a filter or an ideal, and $\preceq$ is either $\supseteq$ or $\subseteq$ respectively\footnote{We denote by $A\subseteq^*B$ if and only if $A\setminus B$ is bounded in a cardinal which is understood from the context.}. One remarkable theorem due to Todorcevic \cite{TodorcevicDirSets85} is that there are only $5$ distinct cofinal types of size at most $\aleph_1$ which are provably different, but many cofinal types of cardinality $\mathfrak{c}$.

The Tukey order on ultrafilters was first considered by J. Isbell \cite{Isbell65} in the 60's. Isbell discovered a combinatorical criterion for the maximality of ultrafilters in the Tukey order. This was later generalized to measurable cardinals in \cite{TomNatasha}.
\begin{theorem}[Isbell]
Let $U$ be an ultrafilter on $\lambda$. The following are equivalent:
\begin{enumerate}
    \item $[2^\lambda]^{<\omega}\leq_T U$.
    \item $\mathbb{P}\leq_TU$ for every directed poset $\mathbb{P}$ of size $2^\lambda$.
    \item There is a subset $\mathcal{A}\subseteq U$ of cardinality $2^\lambda$, such that the intersection of every infinite subset of $\mathcal{A}$ is not in the ultrafilter.
\end{enumerate}
\end{theorem}
Then Isbell \cite{Isbell65} and independently Juhasz \cite{Juhasz67} constructed ultrafilters meeting the above criterion, using long independent families. There have been several constructions of such ultrafilter (see for example \cite{tomFanxin, DowZhou, kunen1972}).

The combinatorical criterion in $(3)$ was independently studied in the 70's by F. Galvin \cite{MR0369081} and A. Kanamori \cite{Kanamori1978}, under different names, as a weak form of regularity of ultrafilters. 
\begin{definition}[Kanamori]
    An ultrafilter $U$ on $\kappa$ is $(\lambda,\mu)$-cohesive if for every $\mathcal{A}\in [U]^\lambda$, there is $\mathcal{B}\in[\mathcal{A}]^\mu$ such that $\bigcap \mathcal{B}\in U$.
\end{definition}
To see the translation, note that Isbell's criterion $(3)$ translates to $U$ not being $(2^\lambda,\omega)$-cohesive. 
In recent developments in Prikry-type forcing theory due to the author, Gitik, Garti and Poveda \cite{TomMoti,partOne,Parttwo,GitDensity,bgp}, cohesiveness was used under (yet) a different name- the \textit{Galvin property}- to characterize certain intermediate models; the statement that $U$ is $(\lambda,\mu)$-cohesive is denoted there by $\text{Gal}(U,\mu,\lambda)$.
These results led to the investigation of the Galvin property at the realm of measurable cardinals \cite{GalDet,Non-GalvinFil,bgs,OnPrikryandCohen,ghhm,GartiTilt}. 

The author and Dobrinen \cite{TomNatasha,TomNatasha2}  made the connection between the two parallel research streams, and developed the basic framework to study the Tukey order on measurable cardinals, generalizing results of Dobrinen and Todorcevic \cite{Dobrinen/Todorcevic11,Dobrinen/Todorcevic14} to the measurable contexts, but also discovering surprising discrepancies between the two. 

Unlike other well-studied orders on ultrafilters, the Tukey order lacks an ultrapower characterization. This makes our understanding of the Tukey order somehow limited, especially when large cardinals are involved or under other axiomatic systems such as canonical inner models or under the Ultrapower Axiom (UA)  \cite{GoldbergUA}. This was pointed out by the author and Goldberg in \cite{TomGabe24}. The first result of this paper provides such an ultrapower characterization of the Tukey order.
\begin{theorem}
    Let $U$ be an ultrafilter and $\mathbb{P}$ any directed set. The following are equivalent:
    \begin{enumerate}
        \item $\mathbb{P}\leq_T U$.
        \item There is a thin cover\footnote{See Definition \ref{def: Thin cover}} $X\in M_U$ of $j_U''\mathbb{P}$.
    \end{enumerate}
\end{theorem}
We then use this characterization to measure the distance between the Tukey order and other orders on ultrafilters such as the Rudin-Keilser order, the Ketonen order. We believe that such a characterization can be useful in determining the structure of the Tukey order on $\sigma$-complete ultrafilters under UA, or at least in the Mitchell models of the form $L[\vec{U}]$, where $\vec{U}$ is a coherent sequence of normal ultrafilter. We also apply our characterization to give a simple ultrapower characterization of cohesiveness

In the second part of this paper, we improve two results from \cite{Kanamori1978}:
\begin{theorem}
    Suppose that $U$ is $(\lambda,\lambda)$-cohesive and $W$ is $(\lambda,\kappa)$ cohesive for $\kappa\leq\lambda$,  or vise versa. Then $U\cdot W$ is $(\lambda,\kappa)$-cohesive.
\end{theorem}
Kanamori proved \cite[Proposition 2.3]{Kanamori1978} the special case where $\lambda=\mu=\omega_1$ and $U,W$ are ultrafilters on $\omega$. He also pointed out that his argument does not generalize to other cardinals. The theorem follows from the author and Dobrinen's simple formulas \cite{TomNatasha,TomNatasha2} for the Tukey-type of Fubini product of $\kappa$-complete ultrafilters on a measurable cardinal $\kappa$ and for certain ultrafilters on $\omega$.

The second result we would like to improve is the following:
\begin{theorem}[{\cite[Theorem 1.2(2)]{Kanamori1978}}]\label{thm:Kanamoti cohesive kappa+}
    Assume $2^\kappa=\kappa^+$. Any uniform ultrafilter over $\kappa$ is not $(\kappa^+,\kappa^+)$-cohesive. 
\end{theorem} 
In light of this result, the following question is natural:
\begin{question}[Kanamori]\label{Kanamoris question}
    Is it consistent that there is a measurable cardinal carrying a $\kappa$-complete ultrafilter which is $(\kappa^+,\kappa^+)$-cohesive?
\end{question}
To formulate our first theorem, we define The \textit{character} of an ultrafilter $U$, as the cardinal: 
    $$\mathfrak{ch}(U)=\min\{|\mathcal{B}|\mid \mathcal{B}\text{ generates U}\}$$
    Where $\mathcal{B}\subseteq U$ generates $U$ (or forms a base for $U$) if for every $X\in U$ there is $b\in \mathcal{B}$ such that $b\subseteq X$.
\begin{theorem}\label{thm:improve Kanamori cohesive}
    Any uniform ultrafilter $U$ over any cardinal $\kappa$ is $(cf(\mathfrak{ch}(U)),cf(\mathfrak{ch}(U)))$-cohesive.
\end{theorem}
Now Kanamori's theorem \ref{thm:Kanamoti cohesive kappa+} is a special case of the Theorem \ref{thm:improve Kanamori cohesive}, since if $2^\kappa=\kappa^+$, then $\mathfrak{ch}(U)=\kappa^+$.

Theorem \ref{thm:improve Kanamori cohesive} is optimal when $\mathfrak{ch}(U)$ is regular, as for every regular $\lambda>\mathfrak{ch}(U)$, $U$ is $(\lambda,<{\lambda})$-cohesive, and for singulars, $(\lambda,<{\lambda})$-cohesive; that is, for every $\mu<\lambda$, $U$ is $(\lambda,\mu)$-cohesive. This leaves an intriguing case when $\mathfrak{ch}(U)$ is singular which we do not address in this paper. Nonetheless, note that this does not mean that $U$ is not $(\lambda',\lambda')$-cohesive for $\lambda'<\mathfrak{ch}(U)$. We then investigate what is known as the \textit{point spectrum} of an ultrafilter, denoted here by $Sp_T(U)$, which consists of all regular cardinal $\lambda$ such that $\lambda\leq_T U$. This was considered before for general directed sets by Isbell \cite{Isbell65} and earlier by Schmidt \cite{Schmidt55}, and recently in connection to pcf theory by Gartside and Mamatelashvili \cite{GARTSIDE2021102873}, and Gilton \cite{Gilton}. For ultrafilters, this was indirectly addressed in \cite{bgp} by Garti, Poveda and the author. We will use this to answer Question \ref{Kanamoris question}.
 
As it investigates cofinal types, the Tukey order is highly connected to $\mathfrak{ch}(U)$ and the  \textit{generalized ultrafilter number}, which is defined for an infinite cardinal $\kappa\geq\omega$ by
$$\mathfrak{u}_\kappa=\min\{\mathfrak{ch}(U)\mid U\text{ is a uniform ultrafilter over }\kappa\}.$$
The most studied instance is $\mathfrak{u}_\omega=\mathfrak{u}$ also known as the \textit{ultrafilter number}. It is now a long-standing open problem whether it is consistent that $\mathfrak{u}_{\omega_1}<2^{\omega_1}$.
The main technique to separate the ultrafilter number from the continuum is to iterate Mathias forcing and create an ultrafilter with a $\subseteq^*$-decreasing generating sequence (see Definition \ref{definition: generating sequence}). This technique does not generalize to higher cardinals, but some variation of it was used to show the consistency of $\mathfrak{u}_\kappa<2^\kappa$ for measurable cardinals $\kappa$ starting from a supercompact cardinal \cite{BROOKETAYLOR201737}. The following is completely open:
\begin{question}
    Is the consistency strength of $\mathfrak{u}_\kappa<2^\kappa$ on a measurable cardinal higher than $o(\kappa)=\kappa^{++}$?
\end{question}
Note that  $o(\kappa)=\kappa^{++}$ is a lower bound since we must violate GCH at a measurable cardinal. 

The technique of obtaining long generating sequence of ultrafilter is tightly related to the generalization of $p$-points considered by Kunen:
\begin{definition}[Kunen]\label{def:Plambda point}
Let $U$ be an ultrafilter. $U$ is called a $P_\lambda$-point if for any $\l A_\alpha\mid \alpha<\mu\r\subseteq U$, where $\mu<\lambda$, there is $A\in U$ such that $A\subseteq^* A_\alpha$ for every $\alpha<\mu$.
\end{definition}
Hence, when $U$ is $\kappa$-complete over $\kappa$, $U$ is a $p$-point precisely when it is a $P_{\kappa^+}$-point. This can of course be formulated in terms of general topological spaces by saying that a point $x$ is a $P_\lambda$-point if every less than $\lambda$ many open neighborhoods of $x$ contain a common open neighborhood. Then an ultrafilter $U$ on $\kappa$ is a $P_\lambda$ point iff it is such in the topological space $\beta\kappa\setminus\kappa$.

We use a refinement of the point spectrum, which we call the \textit{depth Spectrum}, and define the \textit{depth of an ultrafilter} to connect $P_\lambda$-points, Kanamori's question \ref{Kanamoris question}, and strong generating sequences in terms of consistency strength. For precisely we prove that the following are equiconsistent:
\begin{enumerate}
    \item There exists a $P_{\kappa^{++}}$-point.
    \item There is a $\kappa$-complete ultrafilter which is $(\kappa^+,\kappa^+)$-cohesive ultrafilter.
    \item There exists an ultrafilter on $\kappa$ with a strong generating sequence of length $\kappa^{++}$.
\end{enumerate}
It will be clear later that the depth spectrum is to the order $(U,\supseteq^*)$ what the decomposability spectrum of Chang and Keisler is to $(U,\supseteq)$ (i.e. to
completeness).

In the last section, we prove some consistency results regarding the point spectrum of a $\kappa$-complete ultrafilter over $\kappa>\omega$. In particular, we provide a calculation of the point and depth spectrum in the Cohen extension. To round up the picture, relying on the work of Gitik and the above connection, we conclude that a positive answer to Kanamori's question cannot be achieved by the n\"{a}ive lower bound of $o(\kappa)=\kappa^{++}$:
\begin{theorem}
    Suppose that there is a $(\kappa^+,\kappa^+)$-cohesive ultrafilter, then it is consistent that there is a measurable cardinal $\kappa$ with $o(\kappa)\geq \kappa^{++}+\kappa$.
\end{theorem}
Finally, we use the Extender-based Radin forcing of Merimovich \cite{CarmiRadin} to exhibit how to obtain measurable cardinals admitting a long decreasing sequences of clubs without a pseudo intersection. In particular, we are able to add values above $\kappa^+$ to $Sp_{dp}(U)$ starting much lower than a supercompact cardinal.
The structure of this paper is as follows:
\begin{itemize}
    \item In Section~\S\ref{Section: Cover Characterizaiton}: we prove our characterization of the Tukey order in terms of the ultrapower, and deduce some corollaries.
    \item In Section~\S\ref{Sec: Cohesivness and Kanamori's theorem 1} we characterize cohesivness in therms of ultrapowers and improve \cite[Proposition 2.3]{Kanamori1978}.
    \item  In Section~\S\ref{Sec: depth and point} we explore the point and depth spectrum of an ultrafilter to improve Theorem \ref{thm:Kanamoti cohesive kappa+} and to address Question \ref{Kanamoris question}.
    \item In Section\ref{Section: Consistency}, we present several consistency results relevant to the results of the other sections.
\end{itemize}
\subsection*{Notations \& global assumptions}

Our notations are standard for the most part. An ultrafilter $U$ on an infinite set $X$ is a nonempty collection of subsets of $X$ that is closed under intersection and superset, does not contain $\emptyset$, and for every $Y\subseteq X$, with $Y\in U$ or $X\setminus Y\in U$. We say that $U$ is uniform if for every $Y\in U$, $|Y|=|X|$. If $U$ is a filter on $X$ and $f:X\rightarrow Y$ is a map, then $f_*(U)=\{B\subseteq Y\mid f^{-1}(B)\in U\}$ is also an ultrafilter, called the image ultrafilter or the pushforward ultrafilter. Many properties of $U$ are inherited by $f_*(U)$. 

For two ultrafilters $U,V$ on $X,Y$ respectively, we say $U$ is Rudin-Keisler reducible to $V$, denoted $U\leq_{RK}V$, if there is a map $f:Y\rightarrow X$ such that $U=f_*(V)$. We call $U$ and $V$ Rudin-Keisler equivalent, denoted $U\equiv_{RK}V$, if there is a bijection $f:X\rightarrow Y$ such that $U=f_*(V)$. It is a standard fact that $U\leq_{RK}V$ and $V\leq_{RK}U$ imply $U\equiv_{RK}V$. 

$[X]^{\kappa},[X]^{<\kappa},[X]^{\leq\kappa}$ denote the sets of all subsets of $X$ of cardinality $\kappa$, less than $\kappa$, at most $\kappa$, respectively.

Let $F$ be a filter on $X$, and $f,g:X\rightarrow\kappa$. We denote by $f\leq_F g$ if $\{x\in X\mid f(x)\leq g(x)\}\in F$ and we say that $f$ is bounded by $g$ mod $F$; variations on this notation such as $f=_F g$ or $f<_F g$ should be self-explanatory. Note that if $f<_F g$ and $F'$ is a filter extending $F$ then $f<_{F'}g$. A function $f$ is bounded mod $F$ if there is $\alpha\in\kappa$ such that $f\leq_F c_\alpha$ where $c_\alpha$ is the constant function $\alpha$. We say that $f$ is unbounded mod $F$ if $f$ is not bounded mod $F$.
Finally, our forcing convention are in Israel style, namely,  $p\leq q$ means $q$ is stronger than $p$.



\section{A characterization of the Tukey order in terms of the ultrapower}\label{Section: Cover Characterizaiton}
Given two directed partially ordered sets $(\mathbb{P},\leq_{\mathbb{P}}),(\mathbb{Q},\leq_{\mathbb{Q}})$ a \textit{Tukey map} or a \textit{Tukey reduction} from $\mathbb{P}$ to $\mathbb{Q}$ is a function $f:\mathbb{P}\rightarrow \mathbb{Q}$ which is unbounded; that is, whenever $\mathcal{A}\subseteq \mathbb{P}$ is unbounded in $\mathbb{P}$, $f''\mathcal{A}$ is unbounded in $\mathbb{Q}$. The Tukey order, denoted by $\leq_T$, is then defined by setting $(\mathbb{P},\leq_{\mathbb{P}})\leq_T(\mathbb{Q},\leq_{\mathbb{Q}})$ iff there is a Tukey map from $\mathbb{P}$ to $\mathbb{Q}$. Schmidt found that the dual of Tukey maps are cofinal maps; A function $f:\mathbb{Q}\rightarrow\mathbb{P}$ is cofinal if for every $\mathcal{B}\subseteq \mathbb{Q}$ cofinal, $f''\mathcal{B}$ is cofinal in $\mathbb{P}$. 
\begin{proposition}[Schmidt duality \cite{Schmidt55}]
    There is a Tukey map $f:\mathbb{P}\rightarrow \mathbb{Q}$ iff there is a cofinal map $g:\mathbb{Q}\rightarrow \mathbb{P}$
\end{proposition}
We will mostly be interested in the Tukey order restricted to directed sets of the form $(F,\supseteq)$ where $F$ is a filter (usually an ultrafilter) ordered by reversed inclusion.
For filters we may always assume that the cofinal map is (weakly) monotone, that is, if $A\subseteq B$ then $f(A)\subseteq f(B)$. For more information regarding the Tukey order restricted to ultrafilters, we refer the reader to N. Dobrinen's survey \cite{DobrinenTukeySurvey15}. 

The goal of this section is to characterize the Tukey order $\mathbb{P}\leq_T U$ for an ultrafilter over $\kappa\geq\omega$ in terms of its ultrapower, and more precisely, in terms of the existence of certain ``covers" of $j_U''\mathbb{P}$ . To do that we will establish a connection between these covers and functions $f:\mathbb{P}\rightarrow U$.

\begin{definition}
    Let $f,g:A\rightarrow P(\kappa)$ for some set $A$. We say that $f=_Ug$ if there is a set $Z\in U$ such that for every $a\in A$, $f(a)\cap Z=g(a)\cap Z$. 
\end{definition}
Given $X\in M_U$, we pick  $Y\in V$ such that $X\subseteq j_U(Y)$. Also pick a representing function $\vec{X}=\l X_\alpha\mid \alpha<\kappa\r$ (so in particular, $j_U(\vec{X})_{[id]_U}=X$), and denote by $f^{\vec{X}}_X:Y\rightarrow P(\kappa)$ the function  defined by $$f^{\vec{X}}_X(y)=\{\alpha<\kappa\mid y\in X_\alpha\}.$$ Note that if $\vec{X}'$ also represents $X$, then there is a set $Z\in U$ such that for every $\alpha\in Z$, $X_\alpha=X'_\alpha$. So for all $y\in Y$, $f^{\vec{X}}_X(y)\cap Z=f^{\vec{X}'}_X(y)\cap Z$, namely $f^{\vec{X}}_X=_Uf^{\vec{X}'}_X$. We let $f_X$ be some representative of this equivalence class. 

In the other direction, any $f:Y\rightarrow P(\kappa)$ induces a set $X_f$ in $M_U$ defined by $$M_U\models X_f=\{y\in j_U(Y)\mid [id]_U\in j_U(f)(y)\}.$$ 
Once again, note that if $f=_Ug$, then $X_f=X_g$.

\begin{proposition}
    For any $X\in M_U$, any choice of $Y$ so that $X\subseteq j_U(Y)$, $X_{f_X}=X$. Also for any $f:Y\rightarrow P(\kappa)$, $f_{X_f}=_Uf$.
\end{proposition}
\begin{proof} Note that
    $X_{f_X}=\{q\in j_U(Y)\mid [id]_U\in j_U(f_{X})(q)\}$, then for every $q\in j_U(Y)$, 
    \begin{align*}
    q\in X_{f_X} &\text{ iff } [id]_U\in j_U(f_X)(q)\\&\text{ iff }[id]_U\in\{\alpha<j_U(\kappa)\mid q\in j_U(\vec{X})_\alpha\}\\&\text{ iff }q\in j_U(\vec{X})_{[id]_U}\\&\text{ iff }q\in X
\end{align*}
For the second part, $X_f$ is represented by  $\vec{X}=\l X_\alpha\mid \alpha<\kappa\r$, where $X_\alpha:=\{p\in Y\mid \alpha\in f(p)\}$. Let $p\in Y$, then 
$$f^{\vec{X}}_{X_f}(p)=\{\alpha<\kappa\mid p\in (X_f)_\alpha\}=\{\alpha<\kappa\mid 
\alpha\in f(p)\}=f(p)$$

Hence $f=f^{\vec{X}}_X=_Uf_X$

\end{proof}

\begin{definition}
    Let $Z\subseteq M_U$. We say that $X\in M_U$ \textit{covers} $Z$ if for every $p\in Z$, $M_U\models p \in X$. 
\end{definition}
If $M_U$ is well-founded (and therefore identified with its transitive collapse), then a cover is just a superset. From now on, we will write $x\in y$ for elements in $M_U$ where we actually mean that $M_U\models x\in y$. Similarly, $A\cap B$ for $A,B\in M_U$ is defined in $M_U$ as the set of all $p$ such that $M_U\models p\in A\wedge p\in B$, and so on.
\begin{claim} For any function $f:Y\rightarrow P(\kappa)$ and  $Z\subseteq Y$,
    $X_f$ covers $j_U'' Z$ iff $f\restriction Z:Z\rightarrow U$.
\end{claim}
\begin{proof}
    \underline{$\Longrightarrow$:} follows easily from Lo\'{s} Theorem. 
    
    \underline{$\Longleftarrow$:} Let $p\in Z$, then $j_U(f)(j_U(p))=j_U(f(p))$. Since $f(p)\in U$, $[id]_U\in j_U(f(p))$, and by the definition of $X_f$, $j_U(p)\in X_f$.

\end{proof}
By the previous claim, we conclude that
\begin{corollary}\label{cor: cover vs U image}
    For any set $X\in M_U$, and any sets $Z\subseteq Y$, $X$ covers $j_U''Z$ iff $f_X\restriction 
    Z:Z\rightarrow U$.
\end{corollary}
\begin{proof}
    Since
    $X=X_{f_X}$, the corollary follows by applying the claim to $f_X$.
\end{proof}

    Now we translate between properties of $f$ and properties of $X$. The first, is unboundedness:
\begin{definition}\label{def: Thin cover}
    Let $U$ be an ultrafilter, and $\mathbb{P}$ a directed set. We say that a set $X\in M_U$ is a \textit{thin cover} of $\mathbb{P}$ if $j_U''P\subseteq X$ and for any unbounded set $\mathcal{A}\subseteq \mathbb{P}$, $j_U(\mathcal{A})\not\subseteq \mathbb{P}$.
\end{definition}
\begin{lemma}
     Let $\mathbb{P}$ be a directed set and $f:Y\rightarrow P(\kappa)$ such that $\mathbb{P}\subseteq Y$. Then $f\restriction \mathbb{P}:\mathbb{P}\rightarrow U$ is unbounded iff $X_f$ is a thin cover.
\end{lemma}
\begin{proof}
    Suppose $f\restriction \mathbb{P}\rightarrow U$ is unbounded. Then by Corollary \ref{cor: cover vs U image}, $X_f$ is indeed a cover. To see it is thin, suppose $j_U(\mathcal{A})\subseteq X_f$, then $$[id]_U\in \bigcap_{B\in j_U(\mathcal{A})}j_U(f)(B)=j_U(\bigcap_{B\in\mathcal{A}}f(B)).$$
    Hence $\bigcap_{B\in\mathcal{A}}f(B)\in U$, which means that $f''\mathcal{A}$ is bounded in $U$. Since $f\restriction \mathbb{P}$ is unbounded, $\mathcal{A}$ must have been bounded. In the other direction, suppose that $X_f$ is a thin cover. Then by Corollary \ref{cor: cover vs U image}, $f\restriction \mathbb{P}:\mathbb{P}\rightarrow U$. To see that $f\restriction \mathbb{P}$ is unbounded, let $\mathcal{A}\subseteq\mathbb{P}$ be unbounded, since $j_U(\mathcal{A})\not\subseteq  X_f$, $$[id]_U\notin \bigcap_{B\in j_U(\mathcal{A})}j_U(f)(B)=j_U(\bigcap_{B\in\mathcal{A}}f(B)).$$
    Hence $f''\mathcal{A}$ is unbounded in $U$.
\end{proof}
\begin{corollary}
    For any cover $X\in M_U$, $X$ is a thin cover iff $f_X\restriction \mathbb{P}$ is unbounded.
\end{corollary}
This gives a characterization of the Tukey order in terms of covers:
\begin{theorem}\label{thm: Tukey-order characterization}
    Let $U$ be an ultrafilter and $\mathbb{P}$ any directed set. The following are equivalent:
    \begin{enumerate}
        \item $\mathbb{P}\leq_T U$.
        \item There is a thin cover $X\in M_U$ of $j_U''\mathbb{P}$.
    \end{enumerate}
\end{theorem}

\begin{corollary}
    Let $U,W$ be ultrafilters. The following are equivalent:
    \begin{enumerate}
        \item $W\leq_T U$.
        \item There is a cover $X\in M_U$ of $j_U''W$ such that if $\bigcap\mathcal{A}\notin W$, then $j_U(\mathcal{A})\not\subseteq X$.
    \end{enumerate}
\end{corollary}

\begin{remark}
    
We are crucially missing a characterization of a cofinal map $g:U\rightarrow W$ in terms of the ultrapower by $W$. 
\end{remark}
When we cover $j_U''W$, it is tempting to require that the cover $X\in M_U$ is a filter. However, as we will further notice, this corresponds to  Tukey maps with additional properties. From now on, we only consider $\mathbb{P}=W$ for some ultrafilter $W$ and our canonical choice of $Y$ would be $P(\kappa)$. So we consider functions $f:P(\kappa)\rightarrow P(\kappa)$.
\begin{definition}
    We say that a function $f:P(\kappa)\rightarrow P(\kappa)$ is:
    \begin{enumerate}
        \item monotone, if $A\subseteq B\Rightarrow f(A)\subseteq f(B)$
        \item semi-additive, if $f(A)\cap f(B)\subseteq f(A\cap B)$.
        \item additive if $f(A)\cap f(B)=f(A\cap B)$.
        \item $\mu$-semi-additive  if for any $\l A_i\mid i<\lambda\r\in [P(\kappa)]^{<\mu}$,  $\bigcap_{i<\mu}f(A_i)\subseteq f(\bigcap_{i<\mu}A_i)$ \item 
 $\mu$-additive if for any $\l A_i\mid i<\lambda\r\in [P(\kappa)]^{<\mu}$, $\bigcap_{i<\mu}f(A_i)=f(\bigcap_{i<\mu}A_i)$.
 \item negative if $f(\kappa\setminus A)=\kappa\setminus f(A)$.
 \item an homomorphism if $f$ is negative and additive.
    
    \end{enumerate}  
\end{definition}
It is not hard to check that $f$ is $\mu$-additive iff it is $\mu$-semi-additive and monotone. We say that a set $X\subseteq P(\kappa)$ is ultra, if for every $A$,  $$A\in X\text{ xor }\kappa\setminus A\in X.$$ The proof of the following propositions is simple.
\begin{proposition}
    \begin{enumerate}
        \item $f$ is monotone $\Rightarrow$ $X_f$ is upwards closed.
    \item $f$ is  semi-additive $\Rightarrow$ $X_f$ is closed under intersection.
    \item $f$ is  additive $\Rightarrow$ $X_f$ is a filter.
    \item $f$ is $\mu$-additive $\Rightarrow$ $X_f$ is a $\mu$-complete filter.
    \item $f$ is negative $\Rightarrow$ $X_f$ is ultra.
    \item $f$ is an homomorphism $\Rightarrow$ $X_f$ is an ultrafilter.
\end{enumerate}
\end{proposition}
    \qed
    \begin{proposition}
        \begin{enumerate}
        \item $X$ is upwards closed$\Rightarrow$ then there is a monotone $f$ such that   $f_X=f$.
    \item $X$ is closed under intersection $\Rightarrow$  there is
    a semi-additive $f$ so that $f_X=_Uf$ is \item $X$ is a filter $\Rightarrow$ there is an additive $f$  so that $f_X=_Uf$.
    \item $X$ is a $\mu$-complete filter $\Rightarrow$ there is a $\mu$-additive $f$  so that $f_X=_Uf$.
    \item $X$ is ultra $\Rightarrow$  there is a negative $f$  so that $f_X=_Uf$.
    \item $X$ is an ultrafilter $\Rightarrow$ there is an homomorphism $f$ such that $f_X=_Uf$.\end{enumerate}
    \end{proposition}
    \qed
\begin{corollary}
    \begin{enumerate}
        \item There is a thin upward closed cover iff there is a monotone Tukey map.
        \item There is a thin cover closed under intersection iff there is a semi-additive Tukey map.
        \item There is a thin filter cover  iff there is an additive Tukey map.
        \item  There is a thin $\mu$-complete filter cover  iff there is a $\mu$-additive Tukey map.
        \item There is an ultra thin cover  iff there is a negative Tukey map.
        \item There is an ultrafilter thin cover  iff there is an homomorphism Tukey map.
\end{enumerate}
\end{corollary}
\qed

By the ultrafilter lemma we get that:
\begin{corollary}
    There is an additive Tukey map iff there is an homomorphism Tukey map.
\end{corollary}

\begin{lemma}
Any cover $X\in M_U$, closed under intersections, such that $j_U(\mathcal{C})\not\subseteq X$ for every $\mathcal{C}\subseteq W$ with $\bigcap C=\emptyset$ must be a thin cover.
\end{lemma}
\begin{proof}
    It remains to see that $X$ does not contain the image of a set $\mathcal{C}$ such that $\bigcap \mathcal{C}\notin W$. Suppose otherwise, then $j_U(\mathcal{C})\subseteq X$. Let $$\mathcal{C}'=\{Y\cap \overline{\bigcap \mathcal{C}}\mid Y\in \mathcal{C}\}.$$ Then $\bigcap \mathcal{C}'=\emptyset$. But since $X$ covers $j_U''W$, and $\overline{\bigcap \mathcal{C}}\in W$, $j_U(\overline{\bigcap \mathcal{C}})\in X$. Also since $X$ is closed under intersection and $j_U(\mathcal{C})\subseteq X$, then $j_U(\mathcal{C}')=\{ Y\cap j_U(\overline{\bigcap \mathcal{C}})\mid Y\in j_U(\mathcal{C})\}\subseteq X$, contradicting the assumption regarding $X$.
\end{proof}
Next, we would like to use the above characterization to measure the distance between the Tukey order and other orders on ultrafilters. Recall that the Rudin-Keisler order is defined as follows:
\begin{definition}
    Let $U,W$ be ultrafilters over $\kappa$ respectively. We say that $W\leq_{RK}U$ if there is $f:\kappa\rightarrow \kappa$ such that $$W=f_*(U)=\{A\subseteq \kappa\mid f^{-1}[A]\in U\}.$$ 
\end{definition}
Given a Rudin-Keisler projection $f:\kappa\rightarrow \kappa$, we can define $F_f:P(\kappa)\rightarrow P(\kappa)$ by $F_f(A)=f^{-1}[A]$. Then we have the following properties:
\begin{enumerate}
    \item $F_f$ is an  $\infty$-additive homomorphism. That is, for any $\mathcal{A}\subseteq P(\kappa)$, $F_f(\bigcap \mathcal{A})=\bigcap F_f''\mathcal{A}$.
    \item $F_f\restriction W:W\rightarrow U$ is Tukey.
\end{enumerate}
\begin{corollary}
    $W\leq_{RK}U$ then $W\leq_T U$
\end{corollary}
Observe that the Rudin-Keisler order is characterized by:

    $$W\leq_{RK}U\text{ iff }\bigcap j_U''W\neq\emptyset.$$
Given any $\alpha\in\bigcap j_U''W$ we can explicitly define the cover $$X=p^{j_U(\kappa)}_\alpha:=\{A\subseteq j_U(\kappa)\mid \alpha\in A\}.$$ A principal ultrafilter is always $\infty$-complete this ultrafilter cover of $j_U''W$.
This simple observation enables us to characterize the Rudin-Keisler order in terms of unbounded maps:
\begin{corollary}
 $W\leq_{RK}U$ iff there is an $\infty$-additive map $f$ such that $f\restriction W$ is a Tukey reduction.
\end{corollary}
\begin{proof}
    From left to right follows from the previous paragraph. In the other direction, if $f$ is such a map, then there is a thin filter cover $X$ of $j_U''W$ which is $\infty$-complete. This means that $\bigcap X\in X$.  Since $X$ is thin, $X$ is a proper filter and thus $\emptyset\notin X$. It follows that $\bigcap X\neq\emptyset$ and in particular $\bigcap j_U''W\neq\emptyset$. It follows that $U\leq_{RK}W$.
\end{proof}
\begin{remark}
    Ther is no hope of characterazing the Tukey order on all ultrafilters using function $f:\kappa\rightarrow \kappa$, as it is consistent that there are $2^{2^\kappa}$-many incomperable ultrafilters in the Tukey order, and since there is always a Tukey-top (i.e. maximal) ultrafilter.
\end{remark}
Next we address continuity of functions $f:P(\kappa)\rightarrow P(\kappa)$.
\begin{definition} A function $f:P(\kappa)\rightarrow P(\kappa)$ is 
continuous if for every $\alpha<\kappa$ there is $\xi_\alpha<\kappa$ such that for every $A\in W$, $f(A)\cap \alpha$ depends only on $A\cap\xi_\alpha$.
$f$ is Lipschitz if we can pick $\xi_\alpha=\alpha$ and super-Lipschitz if $f(A)\cap\alpha+1$ depends only on $A\cap \alpha$.
\end{definition}

\begin{definition}
We say that $X$ concentrates on $A$ if for every $B,C$ with $B\cap A=C\cap A$, $B\in X$ iff $C\in X$.
\end{definition}
\begin{remark}
    If $X$ is upwards closed, and $X$ concentrates on $A\subseteq j_U(\kappa)$, then $A\in X$: Indeed $j_U(\kappa)\in X$ and $A\cap A=A=A\cap j_U(\kappa)$, hence $A\in X$.

If $X$ is moreover a filter then $X$ concentrates on any $A\in X$: Let $B\in X$ and suppose that $B\cap A=C\cap A$, then $B\cap X\in A$ and therefore $C\cap X\in A$ and therefore $C\in A$.

\end{remark}
\begin{proposition}
    Let $U$ be an ultrafilter. If $f:P(\kappa)\rightarrow P(\kappa)$ is continuous with parameters $\xi_\alpha$, then in 
$M_U$, $X_f$ concentrates on $[\alpha\mapsto \xi_{\alpha+1}]_U$. In particular, if $f$ is Lipschitz then $X_f$ concentrates on $[id]_U+1$ and if $f$ is super-Lipschitz then $X_f$ concentrates on $[id]_U$.
\end{proposition}
\begin{proof}
    Suppose that $X\in A$ and $Y\subseteq j_U(W)$ is such that $Y\cap \xi_{[id]_U+1}=X\cap \xi_{[id]_U+1}$, then $j_U(f)(Y)\cap [id]_{U}+1=j_U(f)(X)\cap [id]_U+1$ and therefore $Y\in A$. 
    \end{proof}
    \begin{proposition}
        If $X\in M_U$ concentrates on $\delta<j_U(\kappa)$ then there is a continuous function $f:P(\kappa)\rightarrow P(\kappa)$
    such that $f_X=_Uf$. In particular, if $X$ concentrates on $[id]_U+1$, then $f$ above cab be Lipschitz and if it concentrates on $[id]_U$ then super-Lipschitz. \end{proposition}\begin{proof}
         Suppose that $X$ concentrates on $\delta=[\alpha\mapsto \delta_\alpha]_U$, and let $$Z=\{\alpha<\kappa\mid X_\alpha\text{ concetrates on }\delta_\alpha\}\in U$$
         Let $\vec{X}$ be a representing sequence for $X$ so that for every $\alpha$, $X_\alpha$ concentrates on $\delta_\alpha$, and let $f=f^{\vec{X}}_X$. By the definition of $f^{\vec{X}}_X$, $f(A)=\{\alpha<\kappa\mid A\in X_\alpha\}$. Hence $f$ is continuous with parameters $\delta_\alpha$.
         Indeed,  for every $\beta<\kappa$, if  $X\cap \delta_\beta=Y\cap \delta_\beta$, then for every  $\gamma\leq\beta$ $X\cap \delta_\gamma=Y\cap \delta_\gamma$ and since $A_\gamma$ concentrates on $\delta_\gamma$, we have that $\gamma\in f(X)$ iff $\gamma\in f(Y)$. Hence $f(X)\cap \beta+1=f(Y)\cap \beta+1$.  Finaly $f=f^{\vec{X}}_X=_U f_X$.
\end{proof}
The following theorem was proven in \cite{TomNatasha}. We will need to derive some fine corollaries from the proof, so let us reproduce the proof here:
\begin{theorem}
     If $U$ is a $p$-point ultrafilter, and $f:U\rightarrow W$ is monotone, then there is $X^*\in U$ such that $f\restriction(U\restriction X^*)$ is continuous.
\end{theorem}
\begin{proof}
    Suppose that $U$ is any $p$-point $\kappa$-complete ultrafilter, $W$ is any ultrafilter on $\kappa$, and $f:U\rightarrow W$ is monotone. Let $\pi:\kappa\rightarrow \kappa$ represent $\kappa$ in $M_U$. Then by $p$-pointness we may assume that $\pi$ is almost one-to-one. Denote by $\rho_\alpha=\sup(\pi^{-1}[\alpha+1])$ and fix any sequence $\gamma_\alpha<\kappa$. For $\alpha<\kappa$. let us define a sequence of sets $X_\alpha$:
    For every $s\subseteq \rho_\alpha$ and $\delta<\gamma_\alpha$, if there is a set $Y\in U$ such that $Y\cap\rho_\alpha=s$ and $\delta\notin f(Y)$, pick $Y_{s,\delta}=Y$. Otherwise let $Y_{s,\delta}=\kappa$. Define $$X_\alpha=\bigcap_{s,\delta}Y_{s,\delta}\in U$$
    \begin{claim}
     $X_\alpha$ has the property that for any $Y\in U$, if $Y\setminus \rho_\alpha\subseteq X_\alpha\setminus \rho_\alpha$, $f(Y)\cap \gamma_\alpha=f((Y\cap \rho_\alpha)\cup(X_\alpha\setminus \rho_\alpha))\cap \gamma_\alpha$.
    \end{claim}
    \begin{proof}
        Indeed, $Y\subseteq (Y\cap \rho_\alpha)\cup(X_\alpha\setminus \rho_\alpha)$, so by monotonicity, we get $``\subseteq"$. In the other direction, if $\delta\notin f(Y)\cap \gamma\alpha$, let $s=Y\cap\rho_\alpha$, we have that $s\cup X_\alpha\setminus \rho_\alpha\subseteq Y_{s,\delta}$ and since $\delta\notin f(Y_{s,\delta})$,  it follows again by monotonicity that $\delta\notin f(s\cup X_\alpha\setminus\rho_\alpha)$.
    \end{proof}
    By $p$-pointness, let $X^*=\Delta^*_{\alpha<\kappa}X_\alpha$ be the modified diagonal intersection defined as
    $$\Delta^*_{\alpha<\kappa}X_\alpha=\{\nu<\kappa\mid \forall \alpha<\pi(\nu), \ \nu\in X_\alpha\}\in U.$$ Then for every $\alpha<\kappa$ we have $X^*\setminus\rho_\alpha\subseteq X_\alpha$. Now for every $\alpha<\kappa$ and every $Y\subseteq X^*$ we have that $f(Y)\cap \gamma_\alpha=f((Y\cap \rho_\alpha)\cup (X^*\setminus \rho_\alpha))\cap \gamma_\alpha$. Hence $f\restriction U\restriction X^*$ is continuous.
\end{proof}
Suppose that $U$ is normal, in particular $\rho_\alpha=\alpha+1$ and suppose that $\gamma_\alpha=\alpha+2$. Denote by $2^{<\kappa}$ the set of all binary functions $f$ with $\dom(f)\in\kappa$. Define   $\hat{f}(h)$ for $h\in 2^\alpha$. For $\alpha=0$, $\hat{f}(h)=\emptyset$. For $\alpha+1$ successor, define $$\hat{f}(h)=f((\chi^{-1}(h) \cap X^*)\cup X^*\setminus \dom(h))\cap\alpha+2,$$
where $\chi:P(\kappa)\rightarrow 2^\kappa$ is the map sending $X$ to its indicator function $\chi(X)$.
For limit $\alpha$, define $\hat{f}(h)=\bigcup_{\beta<\alpha}\hat{f}(h\restriction \beta+1)$. This defined $\hat{f}:2^{<\kappa}\rightarrow P(\kappa)$. Now define $f^*:P(\kappa)\rightarrow P(\kappa)$ by $f^*(X)=\bigcup_{\alpha<\kappa}\hat{f}(\chi(X)\restriction\alpha)$.

\begin{proposition} Suppose that $U$ is normal, then
    \begin{enumerate}
        \item if $h_1\subseteq h_2$, then $\hat{f}(h_1)\sqsubseteq \hat{f}(h_2)$.
        \item $f^*$ is Lipschitz continuous.
        \item $f\restriction (U\restriction X^*)\subseteq f^*$.
        \item $f^*$ is monotone.
        \item If $f$ is cofinal then $f^*$ is cofinal.
        \item  If $f$ is unbounded then $f^*$ is unbounded.
\item If $f$ is cofinal and $U\neq W$, then $f^*$ is super-Lipschits.        \end{enumerate}
\end{proposition}
\begin{proof}
(1) We may assume that $\dom(h_1),\dom(h_2)$ are successor cardinals as the limit case follows from the successor case and the definition of $\hat{f}$. Suppose that $h_1\subseteq h_2$ have successor domains $\alpha_1\leq\alpha_2$, then $\chi(h_2)\cap X^*\cap \alpha_1=\chi(h_1)\cap X^*$, we get that $\hat{f}(h_1)=f((\chi(h_1)\cap X^*)\cup 
 X^*\setminus \alpha_1)\cap \alpha_1+1$. Also, taking $Y=(\chi(h_2)\cap X^*)\cup 
 X^*\setminus \alpha_2$, we have that $Y\subseteq X^*$. By the property of $f$, we have that $f(Y)\cap \alpha_1+1=f((Y\cap \alpha_1)\cup X^*\setminus \alpha_1)\cap \alpha_1+1$. Hence
 \begin{align*}
    \hat{f}(h_2)\cap \alpha_1+1&=f((\chi(h_2)\cap X^*)\cup 
 X^*\setminus \alpha_2)\cap \alpha_1+1\\&=f((Y\cap \alpha_1)\cup X^*\setminus \alpha_1)\cap \alpha_1+1\\&=f((\chi(h_2)\cap X^*\cap \alpha_1)\cup (X^*\setminus \alpha_1))\cap \alpha_1+1\\&=f((\chi(h_1)\cap X^*)\cup (X^*\setminus \alpha_1))\cap \alpha_1+1=\hat{f}(h_1).
\end{align*}
 Thus $\hat{f}(h_1)\sqsubseteq \hat{f}(h_2)$. 

 (2) follows from (1). Indeed,  since $f^*(X)$ is the $\sqsubseteq$-increasing union of $\hat{f}(\chi(X)\restriction\alpha)$, we see that $$f^*(X)\cap \alpha= \hat{f}(\chi(X)\restriction \alpha)\cap \alpha,$$ and if $X\cap \alpha=Y\cap \alpha$ then $\chi(X)\restriction\alpha=\chi(Y)\restriction \alpha$.  

 (3) If $Y\subseteq X^*$ and $Y\in U$, then for every $\alpha<\kappa$, $Y\cap X^*\cap \alpha+1=Y\cap \alpha+1$ and $$f(Y)\cap \alpha+1=f((Y\cap \alpha+1)\cup (X^*\setminus \alpha+1))\cap \alpha+1=\hat{f}(\chi(Y)\restriction \alpha+1)=f^*(Y)\cap \alpha+1.$$ Since this is true for every $\alpha$, $f^*(Y)=f(Y)$.

 (4) For every $X\subseteq \kappa$ and for every  $\alpha<\kappa$, $$\hat{f}(\chi(X)\restriction \alpha+1)=f(X\cap X^*\cap \alpha+1)\cup X^*\setminus \alpha+1)\cap\alpha+1.$$ Since $f$ is monotone, we conclude that if $X_1\subseteq X_2$, then for every $\alpha$, $f^*(X_1)\cap \alpha+1\subseteq f^*(X_2)\cap \alpha+1$. resulting in $f^*(X_1)\subseteq f^*(X_2)$. 
 
 (5) If $f$ is cofinal, then $f''U\restriction X^*$ is also cofinal (since $U\restriction X^*$ is cofinal in $U$). By $(3)$, $f^*{}''U\restriction X^*$ is cofinal and therefore $f^*$ has a cofinal image. This together with $(4)$ suffices for $f^*$ to be cofinal.
 
 (6) Suppose that $f$ is unbounded. To see that $f^*$ is unbounded, suppose that $\mathcal{A}\subseteq U$ is unbounded, then also $\mathcal{A}\restriction X^*=\{A\cap X^*\mid X\in \mathcal{A}\}$ is unbounded. Since $f$ is unbounded then $f''\mathcal{A}\restriction X^*$ in unbounded. Now it is not hard to see that for every $Y\subseteq \kappa$, $f^*(Y)=f^*(Y\cap X^*)=f(Y\cap X^*)$ and therefore $f^*{}''\mathcal{A}$ is also unbounded.

(7) First note that by the proof of $(2)$, for $\alpha$ successor we get even more, that $f^*(X)\cap \alpha+1=\hat{f}(\chi(X)\restriction\alpha)$, and therefore if $X\cap\alpha=Y\cap \alpha$ then $f^*(X)\cap \alpha+1=f^*(X)\cap \alpha+1$. Now let $X_0\in U$ be such that $X_0^c\in W$. Since $f$ is cofinal, there is $Y$ such that $f(Y)\subseteq X_0^c$ and we may take our $X^*$ so that $X^*\subseteq X_0\cap Y$ and in particular $f(X^*)\subseteq X_0^c$. Note that for every $h$, $\hat{f}(h)\subseteq f(X^*)$ and therefore for every $X$, $f^*(X)\subseteq f(X^*)$. Let $\alpha<\kappa$ and suppose that $X,Y$ are such that $X\cap \alpha=Y\cap \alpha$. By $(2)$, $f^*(X)\cap \alpha=f^*(Y)\cap \alpha$. Let us split into cases. If $\alpha\in X_0$, then $\alpha\notin f^*(X)$ and $\alpha\notin f^*(Y)$. Hence $f^*(X)\cap\alpha+1=f^*(Y)\cap\alpha+1$. If $\alpha\notin X_0$, then $\alpha\notin X^*$, hence $\alpha\notin X\cap X^*$ and $\alpha\notin Y\cap X^*$. It follows that $$X\cap X^*\cap\alpha+1=X\cap X^*\cap\alpha=Y\cap X^*\cap \alpha= Y\cap X^*\cap\alpha+1$$
We conclude that $$f^*(X)\cap \alpha+1=f^*(X\cap X^*)\cap\alpha+1=f^*(Y\cap X^*)\cap\alpha+1=f^*(Y)\cap \alpha+1.$$
\end{proof}
    


\begin{corollary}
    If $W\leq_T U$, and $U$ is normal and $W\neq U$, then there is an upwards closed cover $X\in M_W$ of $j_U''W$ such that $X$ concentrates on $[id]_U$. 
\end{corollary}
The above is closely related to the Ketonen order. The reason is that
the Ketonen ordered is defined by $U<_{\mathbb{K}}W$ if there is a cover $X$ such that $X$ is an ultrafilter concentrating on $[id]_U$.  

Since we lack the characterization of cofinal maps in the ultrapower, there are some properties of the cover $X$ we are missing. These properties might be useful in an attempt to prove that normal ultrafilters are Tukey minimal.
\begin{question}
    Is there a property of the cover $X_f$ that is equivalent to $f$ being cofinal? How about monotone continuous and cofinal?
\end{question}
    
\section{On Cohesive ultrafilters}\label{Sec: Cohesivness and Kanamori's theorem 1}
Recall that an ultrafilter $U$ is $(\lambda,\mu)$-cohesive (or alternatively, $\text{Gal}(U,\mu,\lambda)$ holds) if for any $\mathcal{A}\in[U]^\lambda$ there is $\mathcal{B}\in[\mathcal{A}]^\mu$ such that $\bigcap\mathcal{B}\in U$. 
The results of the previous section provide an elegant characterization of the cohesiveness in terms of the ultrapower:
\begin{theorem}\label{thm:Characterization of Galvin}
    Let $U$ be an ultrafilter over $\kappa\geq\omega$, and $\mu\leq \lambda$ be any cardinals. Then the following are equivalent:
    \begin{enumerate}
        \item $U$ is $(\lambda,\mu)$-cohesive.
        \item Any cover $X\in M_U$ of $j_U''\lambda\subseteq X$ contains a set of the form $j_U(Y)$ for some  $Y\in[\lambda]^\mu$.
    \end{enumerate}
\end{theorem}
\begin{proof}
    Assume that $U$ is $(\lambda,\mu)$-cohesive and let $X=[\beta\mapsto X_\beta]_U$ be any cover as in $(2)$. For every $i<\lambda$ let $A_i=\{\beta<\kappa\mid i\in X_\beta\}$, then by Lo\'{s} Theorem, $A_i\in U$.  By $(\lambda,\mu)$-cohesiveness, there is $Y\in[ \lambda]^\mu$ such that $A^*=\bigcap_{i\in Y}A_i\in U$. It follows that for every $\beta\in A^*$, $Y\subseteq X_\beta$ and therefore (again by Lo\'{s}) $j_U(Y)\subseteq X$.

    In the other direction, suppose that $(2)$ holds and let us prove that $U$ is $(\lambda,\mu)$-cohesive. Let $\l A_i\mid i<\lambda\r\subseteq U$. Working in $M_U$, consider $$j_U(\l A_i\mid i<\lambda\r)=\l A'_i\mid i<j_U(\lambda)\r$$ and define $X=\{\beta<j_U(\lambda)\mid [id]_U\in A'_\beta\}\in M_U$. Then, since $A'_{j_U(\beta)}=j_U(A_\beta)$, $[id]_U\in A'_{j_U(\beta)}$ and $j_U(\beta)\in X$. It follows that $X$ covers $j_U''\lambda$. By $(2)$ there is $Y\in [\lambda]^\mu$ such that $j_U(Y)\subseteq X$ and therefore $\bigcap_{i\in Y}A_i\in U$. To see this, note that  $$j_U(\bigcap_{i\in Y}A_i)=\bigcap_{i\in j_U(Y)}A'_i$$ and since $j_U(Y)\subseteq X$, $[id]_U\in A'_i$ for all $i\in j_U(Y)$ hence $[id]_U\in j_U(\bigcap_{i\in Y}A_i)$
\end{proof}
\begin{remark}
     \begin{enumerate}
         \item Requiring that the cover $X$ contains a set of the form $j_U(Y)$ where $Y$ has size $\mu$ is equivalent to representing $X=[\beta\mapsto X_\beta]_U$ and requiring the existence of $A\in U$ such that $\bigcap_{\beta\in A}X_\beta$ has size $\mu$.
         \item Condition $(2)$ above can be replaced with the following: Any $X\in M_U$ such that (in $V$) $|X\cap j_U''\lambda|=\lambda$ contains  a set of the form $j_U(Y)$ for some $Y\in[\lambda]^\mu$.
         The reason is that the existence of a bijection $\varphi:\lambda\rightarrow j_U^{-1}[X]$ enables to convert the set $X$ to a full cover of $j_U''\lambda$.
     \end{enumerate}
\end{remark}
Let us now reproduce the characterization of Tukey-top ultrafilters using Theorem  \ref{thm:Characterization of Galvin} and Theorem \ref{thm: Tukey-order characterization}
\begin{corollary}
    Let $U$ be an ultrafilter over $\kappa\geq\omega$, and $\mu\leq\lambda$ are such that $\lambda^{<\mu}=\lambda$. The following are equivalent:
    \begin{enumerate}
        \item For any $\mu$-directed set $\mathbb{P}$, such that $|\mathbb{P}|\leq\lambda$, $\mathbb{P}\leq_T (U,\supseteq)$.
        \item $([\lambda]^{<\mu},\subseteq)\leq_T (U,\supseteq)$.
        \item $U$ is not $(\lambda,\mu)$-cohesive.
    \end{enumerate}
\end{corollary}
\begin{proof}
    Clearly, $(1)$ implies $(2)$. To see that $(2)$ implies $(3)$, by Theorem \ref{thm: Tukey-order characterization} there is a thin cover $X$ of $j_U''[\lambda]^{<\mu}$. Define $X'=\{Y\mid \{Y\}\in X\}\in M_U$ Then $j_U''\lambda\subseteq X'$. Note that if $I\in [\lambda]^\mu$, then $I^*=\{\{i\}\mid i\in I\}$ is unbounded in $[\lambda]^{<\mu}$ and therefore there $j_U(I^*)\not\subseteq X$. therefore, there is $i\in j_U(I)$ such that $\{i\}\notin X$, namely, $i\notin X'$. So $X'$ is a thin cover of $j_U''\lambda$ and by Theorem \ref{thm:Characterization of Galvin}, $U$ is not $(\lambda,\mu)$-cohesive. To see that $(3)$ implies $(1)$, Let $\mathbb{P}$ be any $\mu$-directed set of size $\leq\lambda$ and let $f:\mathbb{P}\rightarrow \lambda$ be injective. By Theorem \ref{thm:Characterization of Galvin} find a thin cover $X$ of $j_U''\lambda$. Then $X'=j_U^{-1}(f)[X]$ is a thin cover of  $j_U''\mathbb{P}$.
\end{proof}

By Galvin's Theorem \cite{MR0369081}, every normal (or even $p$-point) ultrafilter is $(\kappa^+,\kappa)$-cohesive. Hence we get the following:
\begin{corollary}
    If $U$ is normal (or even a $p$-point), then whenever $X\in M_U$ covers $j_U''\kappa^+$, there is $Y\in [\kappa^+]^\kappa$ such that $j_U(Y)\subseteq X$.
\end{corollary}
An improvement of Galvin's theorem can be found in \cite{SatInCan} to iterated sums of $p$-points. So the above corollary holds in this generality as well.


Kanamori proved that if $\{U\}\cup\{W_n\mid n<\omega\r\}$ is  a set of $(\omega_1,\omega_1)$-cohesive ultrafilters on $\omega$, then $\sum_U W_n$ is $(\omega_1,\omega)$-cohesive. He then says that this does not generalize to $\kappa>\omega$. We will prove that to some extent his result do generalize to $\kappa>\omega$, and that on measurable cardinals we can even say a bit more. We say that a directed set $\mathbb{P}$ is $(\lambda,\mu)$-cohesive, if for every $\l p_\alpha\mid\alpha<\lambda\r\subseteq\mathbb{P}$ there is $I\in[\lambda]^\mu$ such that $\{p_i\mid i\in I\}$ is bounded. It is not hard to see that cohesiveness is an invariant of the Tukey order: 
\begin{fact}
    If $\mathbb{P}\leq_T\mathbb{Q}$ and $\mathbb{Q}$ is $(\lambda,\mu)$-cohesive, then $\mathbb{P}$ is $(\lambda,\mu)$-cohesive. So if $\mathbb{P}\equiv_T\mathbb{Q}$, then $\mathbb{Q}$ is $(\lambda,\mu)$-cohesive if and only if $\mathbb{P}$ is $(\lambda,\mu)$-cohesive.
\end{fact}
In fact, similar to ultrafilters, it is possible to show that if $\lambda^{<\mu}=\lambda$, $\mathbb{P}$ is not $(\lambda,\mu)$-cohesive exactly when $([\lambda]^{<\mu},\subseteq)\leq_T \mathbb{P}$.
 We will also need the following theorem.
\begin{theorem}[B.-Dobrinen {\cite{TomNatasha}}]\label{thm: Tom Natasha product for complete}
    Let $U,W$ be $\kappa$-complete ultrafilters over $\kappa>\omega$. Then $U\cdot W\equiv_T U\times W$
\end{theorem}
\begin{theorem}\label{thm: improvement of product}
    Let $U,W$ be $\kappa$-complete ultrafilters over $\kappa>\omega$. Suppose that $U$ is $(\lambda,\lambda)$-cohesive and $W$ is $(\lambda,\mu)$-cohesive (or the other way around). Then $U\cdot W$ is $(\lambda,\mu)$-cohesive.
\end{theorem}
\begin{proof}
    First, by Theorem \ref{thm: Tom Natasha product for complete}, $U\cdot W\equiv_T U\times W$. By the previous fact, it suffices to prove that  $U\times W$ is $(\lambda,\kappa)$-cohesive. Given $\l (A_\alpha,B_\alpha)\mid \alpha<\lambda\r$ we need to find $\mu$-many of the pairs which have a lower bound. Indeed since $U$ is $(\lambda,\lambda)$-cohesive,  there is $I\subseteq \lambda$ of size $\lambda$ such that $B^*=\bigcap_{i\in I}B_i\in U$. Applying $(\lambda,\mu)$-cohesiveness to $\l A_i\mid i\in I\r$, find $J\subseteq I$ of size $\mu$ such that $A^*=\bigcap_{j\in J}A_j\in W$. Then $\l \l A_j,B_j\r\mid j\in J\r$ is bounded by $(A^*,B^*)$.
\end{proof}
On $\omega$, it is not true that for every two ultrafilter $U,W$, $U\cdot W\equiv U\times W$. Indeed, Dobrinen-Todorcevic \cite{Dobrinen/Todorcevic11} gave an exmaple (under $\mathfrak{u}<\mathfrak{d}$) of an ultrafilter (even a $p$-point) $U$ which is not Tukey equivalent to its Fubini square. However, by \cite[Cor. 1.9]{TomNatasha2}, if $W\cdot W\equiv_T W$, then $U\cdot W\equiv_TU\times W$ and the above proof works. 
\begin{corollary}\label{cor: the case of omega}
    Fix any two ultrafilters $U,W$ over $\omega$ such $W\cdot W\equiv_T W$. If $U$ is $(\lambda,\lambda)$-cohesive and $W$ is $(\lambda,\mu)$-cohesive then $U\cdot W$ is $(\lambda,\mu)$-cohesive. 
\end{corollary}
In \cite{TomNatasha2}, the class of ultrafilters which satisfies $W\cdot W\equiv_T W$ was investigated, and includes rapid $p$-points (and almost rapid $p$-points- see \cite{Tomcommute}), many instances of generic ultrafilters for $P(\omega)/I$, Milliken-Taylor ultrafilter, and more. 

A general formula for Fubini products was given by  Todorcevic-Dobrinen and Milovich \cite{Dobrinen/Todorcevic11,Milovich12}, but more relevant to out needs, the following upper bound for Fubini sums:
\begin{theorem}\label{thm: formula for omega}
    For any two ultrafilters $U,W_n$ on $\omega$,
    $$\sum_UW_n\leq_T U\times \prod_{n<\omega}W_n.$$
\end{theorem}
Let us reproduce Kanamori's result on $\omega$:
\begin{theorem}
    Suppose that $U$ and $W_n$ are $(\omega_1,\omega_1)$-cohesive for every $n<\omega$. Then $\sum_UW_n$ is $(\omega_1,\omega)$-cohesive.
\end{theorem}
\begin{proof}
    Since $(\lambda,\mu)$-cohesiveness is downwards closed with respect to the Tukey order, by theorem \ref{thm: formula for omega}, it remains to see that $U\times \prod_{n<\omega}W_n$ is $(\omega_1,\omega)$-cohesive. In the next lemma we sill prove that the Cartesian product of $\omega$-many $(\omega_1,\omega_1)$-cohesive directed sets is $(\omega_1,\omega)$-cohesive. 
\end{proof}
\begin{lemma}
    Suppose that $\{\mathbb{P}_i\mid i<\omega\}$ is a countable set of $(\omega_1,\omega_1)$-cohesive directed sets. Then $\prod_{n<\omega}\mathbb{P}_i$ is $(\omega_1,\omega)$-cohesive
\end{lemma}
\begin{proof}
    Let $\l (p_{\alpha,m})_{m<\omega}\mid \alpha<\omega_1\r\subseteq \prod_{m<\omega}\mathbb{P}_m$. Inductively find $X_0\supseteq X_1\supseteq X_2...$ all of size $\omega_1$ such that for every $m$, $\{p_{\alpha,m}\mid \alpha\in X_m\}$ is bounded by $p^0_m$. Choose $\alpha_i\in X_{i}$ so that $i\neq j$ implies $\alpha_i\neq \alpha_j$. Let $p^*_m$ be an extension of $p^0_m$ and $p_{\alpha_i,m}$ for $i< m$, which exists by directedness. Then we claim that $\{ (p_{\alpha_i,m})_{m<\omega}\mid i<\omega\}$ is bounded by $(p^*_m)_{m<\omega}$. Indeed, for every $i<\omega$, and every $m<\omega$, if $i<m$, then by $p^*_m$ was chosen to be an extension of $p_{\alpha_i,m}$. If $i\geq m$, then $\alpha_i\in X_i\subseteq X_m$. Hence $p_{\alpha_i,m}$ is bounded by $p^0_m$ and in turn by $p^*_m$.  
\end{proof}
On $\kappa>\omega$, there is an analog formula: for $U$ an ultrafilter on $\gamma$ and $\l W_\alpha\mid \alpha<\gamma\r$ any ultrafilters, 
$$\sum_U W_\alpha\leq_T U\times \prod_{\alpha<\gamma}W_\alpha.$$
However the argument above does not generalize to arbitrary products, as we might not be able to find longer decreasing sequence in the course of the lemma. 
    \section{Depth and point spectrum of ultrafilters}\label{Sec: depth and point}
    In this section we address the property of $(\lambda,\lambda)$-cohesivness for general $\lambda$.  \begin{claim}\label{claim: cohesive iff Tukey}
        For any infinite cardinal $\lambda$, $U$ is not $(cf(\lambda),cf(\lambda))$-cohesive iff $\lambda\leq_T U$
    \end{claim}
    \begin{proof}
        Note that $\lambda\equiv_T cf(\lambda)$ and since $cf(\lambda)$ is regular,  $[cf(\lambda)]^{<cf(\lambda)}\equiv_T  (cf(\lambda),<)$, and the rest follows from Theorem \ref{thm:Characterization of Galvin}.
    \end{proof}
    The following theorem is due to Kanamori:
\begin{theorem}[Kanamori\cite{Kanamori1978}]
Suppose that $2^\kappa=\kappa^+$, then any ultrafilter $U$ over $\kappa$ is not $(\kappa^+,\kappa^+)$-cohesive.
\end{theorem}
The previous theorem is a special case of the following theorem since under $2^\kappa=\kappa^+$, $\mathfrak{ch}(U)=\kappa^+$ for any uniform ultrafilter on $\kappa$:
\begin{theorem}\label{thm: second kanamori improvment}
    Let $U$ be a uniform ultrafilter over $\kappa$. Then $\mathfrak{ch}(U)\leq_TU$.
\end{theorem}
\begin{proof}
    Let us construct a sequence of length $\mathfrak{ch}(U)$ witnessing that $\mathfrak{ch}(U)\leq_T U$. Let $\mathcal{B}$ be a base for $U$ of size $\theta$, and let us construct a sequence $\l b^*_i\mid i<\mathfrak{ch}(U)\r$ starting with $b^*_0=b_0$. Suppose that $\l b^*_i\mid i<\gamma\r$ has been defined. By minimality of $\mathfrak{ch}(U)$, there is $X\in U$ such that for any $i<\gamma$, $b^*_i\not\subseteq X$. Since $\mathcal{B}$ is a base, there is $b\in \mathcal{B}$ such that $b\subseteq X$, and therefore for any $i<\gamma$, $b^*_i\not\subseteq b$. let $i_\gamma<\mathfrak{ch}(U)$ be minimal such that $b^*_{i_\gamma}$ is not generated by the previous $b_i$'s, and $b^*_\gamma=b_{i_\gamma}$. It is not hard to check that $\l b^*_i\mid i<\mathfrak{ch}(U)\r$ is again a base for $U$. Suppose towards a contradiction that there is $I\subseteq\theta$ unbounded such that $\bigcap_{i\in I}b^*_i\in U$. Then there $j<\theta$ such that $b^*_j\subseteq \bigcap_{i\in I}b^*_i$. Pick any $i\in I$ such that $i>j$, then $b^*_j\subseteq b^*_i$, contradicting the choice of $b^*_i$.

\end{proof}
Kanamori asked \cite[Question 2]{Kanamori1978} the following:
\begin{center}
    Is the existence of a $\kappa$-complete ultrafilter over $\kappa>\omega$ which is $(\kappa^+,\kappa^+)$-cohesive consistent?
\end{center}
    
In this section we will answer Kanamori's question by studying the \textit{point spectrum} and establishing some connections of it to the order $(U,\supseteq^*)$. This order was studied by Milovich \cite{Milovich08} and later by Dobrinen and the author \cite{TomNatasha}.
\subsection{The point spectrum of an ultrafilter}
Define the \textit{point spectrum} of an ultrafilter $U$ by $$Sp_{T}(U)=\{\lambda\in Reg\mid \lambda\leq_T (U,\supseteq)\}.$$ 
By Claim \ref{claim: cohesive iff Tukey}, we also have 
$$Sp_{T}(U)=\{\lambda\in Reg\mid U\text{ is  not }(\lambda,\lambda)\text{-cohesive}\}.$$
Isbell \cite{Isbell65} and independently Juh\'{a}sz \cite{Juhasz67} proved that cardinals such that $\kappa^{<\kappa}=\kappa$ always admits a uniform ultrafilter $U$ which is not $(2^\kappa,\omega)$-cohesive, and therefore $Sp(U)=Reg\cap [\omega,2^\kappa]$. Moreover, if for example, $\kappa$ is $\kappa$-compact (or even less- see \cite{TomGabe24}) there is always a $\kappa$-complete ultrafilter, and even one which extends the club filter, which is not $(2^\kappa,\kappa)$-cohesive. Such ultrafilter in particular are not $(\lambda,\lambda)$-cohesive for any regular $\kappa\leq\lambda\leq 2^\kappa$.


\begin{definition}
    An ultrafilter $U$ is $(\lambda,\mu)^*$-cohesive if the directed set $(U,\supseteq^*)$ $(\lambda,\mu)$-cohesive; that is, if for every sequence $\l A_\alpha\mid\alpha<\lambda\r\subseteq U$ there is $I\in[\lambda]^\mu$ such that $\{A_i\mid i\in I\}$ admits a pseudo intersection in $U$
\end{definition}
Similar to the usual cohesiveness characterization, we have the following:
\begin{proposition}
    Let $U$ be an ultrafilter. The following are equivalent:
    \begin{enumerate}
        \item For any $\mu$-directed poset $\mathbb{P}$, such that $|\mathbb{P}|\leq\lambda$, $\mathbb{P}\leq_T (U,\supseteq^*)$.
        \item $([\lambda]^{<\mu},\subseteq)\leq_T (U,\supseteq^*)$.
        \item $U$ is not $(\lambda,\mu)^*$-cohesive.
    \end{enumerate}
\end{proposition}
Define
$$Sp^*_T(U)=\{\lambda\in Reg\mid \lambda\leq_T (U,\supseteq^*)\}=\{\lambda\in Reg\mid U\text{ is not }(\lambda,\lambda)^*\text{-cohesive}\}.$$
It is easy to see that $(U,\supseteq^*)\leq_T (U,\supseteq)$ and that $(\lambda,\mu)$-cohesivness implies $(\lambda,\mu)^*$-cohesivness. Moreover the other implication is usually true as well:
\begin{lemma}
    Let $U$ be a uniform ultrafilter over $\kappa$ and a cardinal $\mu$ such that $cf(\mu)\neq\kappa$. If $U$ is  $(\lambda,\mu)^*$-cohesive then $U$ is $(\lambda,\mu)$-cohesive. 
\end{lemma}
\begin{proof} Assume that $U$ is $(\lambda,\mu)^*$-cohesive and
    let $\l X_i\mid i<\lambda\r\subseteq U$. By assumption, there is $A\in U$ and $I\in [\lambda]^\mu$ such that for every $i\in I$, $A\subseteq^* X_i$.
For each $i\in I$, let $\xi_i<\kappa$ be such that $A\setminus \xi_i\subseteq X_i$. Let us split into cases. If $\mu>\kappa$, then there is $I'\in [I]^\mu$ and $\xi^*<\kappa$ such that for every $i\in I'$, $\xi_i=\xi^*$. If $cf(\mu)<\kappa$, and let $\l 
 \mu_i\mid i<cf(\mu)\r$ be regular cardinals different from $\kappa$ converging to $\mu$. Write $I=\biguplus_{i<cf(\mu)}I_i$ where $|I_i|=\mu_i$, for each $i<cf(\mu)$. If $\mu_i<\kappa$, we let $\eta_i=\sup_{j\in I_i}\xi_j$, and if $\mu_i>\kappa$, since it is regular, we can apply the previous part to find $J_i\subseteq I_i$, $|J_i|=\mu_i$ such that for every $j\in J_i$, $\xi_j=\eta_i$. Then we can take $\xi^*=\sup_{i<cf(\mu)}\eta_i<\kappa$ and we let $I'=\bigcup_{i<cf(\mu)}J_i$. In any case, $A\setminus\xi^*\subseteq \bigcap_{i\in I'} X_i$. By uniformity, $A\setminus \xi^*\in U$ and therefore $\bigcap_{i\in I'}X_i\in U$.
\end{proof}
\begin{corollary}\label{Cor: Spec^* is enough} If $cf(\lambda)\neq\kappa$, then $U$ is $(\lambda,\lambda)$-cohesive iff $U$ is $(\lambda,\lambda)^*$-cohesive. In particular,
$$Sp_{T}(U)\setminus \{\kappa\}=Sp^*_T(U)\setminus\{\kappa\}$$
\end{corollary}
\begin{remark} For any uniform ultrafilter over $\kappa$ will have $\kappa\in Sp_{T}(U)$. However $U$ can be $(\kappa,\kappa)^*$-cohesive if for example $U$ is a $p$-point. In fact, for $\kappa$-complete ultrafilters $U$ over $\kappa$, $U$ is a $p$-point iff $U$ is   $(\kappa,\kappa)^*$-cohesive. Or in other words, for $\kappa$-complete ultrafilters, $U$ is not a $p$-point iff $Sp_T(U)=Sp^*_T(U)$.
\end{remark}
\begin{definition}
    Let $U$ be an ultrafilter. Let the \textit{Depth Spectrum of U} be denoted by $Sp_{dp}(U)$ and defined to be the set of all regular cardinal lengths $\theta$ of sequences $\l A_i\mid i<\theta\r\subseteq U$ which are $\supseteq^*$-decreasing and have no measure one pseudo intersection in $U$. Define the \textit{Depth} of $U$ to be $\mathfrak{dp}(U)=\min(Sp_{dp}(U))$.
\end{definition}
It is not hard to prove using Zorn's lemma that $Sp_{dp}(U)\neq\emptyset$ and therefore $\mathfrak{dp}(U)$ is well defined.
Note that if $\theta$ is singular, and $\l A_i\mid i<\theta\r$ is $\supseteq^*$-decreasing with no $\supseteq^*$-bound, then $cf(\theta)\in Sp_{dp}(U)$. Hence $\mathfrak{dp}(U)$ is a regular cardinal. 

Suppose that $\kappa$ is $\kappa$-compact\footnote{That is, every $\kappa$-cpomplete filter can be extended to a $\kappa$-complete ultrafilter.}, and  let $\l X_i\mid i<\theta\r$ be a tower. Since any tower has the $<\kappa$-intersection property, there is a $\kappa$-complete ultrafilter $U$ such that for every $i$, $X_i\in U$. Hence $\theta\in Sp_{dp}(U)$. In particular, the tower number $\mathfrak{t}_\kappa\in Sp_{dp}(U)$ for some $U$. Since $U$ is uniform, then it has to be that $\mathfrak{t}_\kappa=\mathfrak{dp}(U)$ as any pseudo intersection in $U$ must have size $\kappa$. 

\begin{proposition}
   $Sp_{dp}(U)\subseteq Sp_{T}(U)$.
\end{proposition}
\begin{proof}
Fix a witnessing sequence $\l A_\alpha\mid \alpha<\theta\r\subseteq U$ is $\subseteq^*$-decreasing. If $U$ was $(\theta,\theta)$-cohesive, then there would have been $I\in [\theta]^{\theta}$ such that $A\in U$ is a pseudo intersection of $\{A_i\mid i\in I\}$. We claim that $A$ is a pseudo intersection for the entire sequence. Indeed, let $\alpha<\theta$, then there is $\alpha'\in I$ such that $\alpha'\geq\alpha$. Hence $A\subseteq^* A_{\alpha'}\subseteq^* A_\alpha$. Contradiction.
\end{proof}

The most general setup to examine the point spectrum is the Galois-Tukey connections 
(see for example Blass's Chapter in \cite{BlassHandbook}).  However, this generality will not contribute to our specific interest in ultrafilters. Given a directed set $\mathbb{P}$, the \textit{lower character} $l(\mathbb{P})$ is the smallest size of an unbounded family in $\mathbb{P}$, while the \textit{upper character} $u(\mathbb{P})$ is the smallest size of a cofinal subset of $\mathbb{P}$. In our case, where $\mathbb{P}=(U,\supseteq)$ or $\mathbb{P}=(U,\supseteq^*)$ we get:
\begin{proposition}
    For any uniform ultrafilter $U$, we have:
    \begin{enumerate}
        \item $u((U,\supseteq))=u((U,\supseteq^*))=\mathfrak{ch}(U)$. 
        \item $l(U,\supseteq)=crit(j_U)$ is the completeness degree of the ultrafilter $U$.
        \item $l(U,\supseteq^*)=\mathfrak{dp}(U)$
    \end{enumerate}
    We have that $crit(j_U)\leq\mathfrak{dp}(U)\leq\mathfrak{ch}(U)$.
\end{proposition}
\begin{proof}
     $(1),(2)$ are well known facts. To see $(3)$, by minimality, $l(U,\supseteq^*)\leq \mathfrak{dp}(U)$. The other direction follows from the next simple lemma which implies that  for every sequence of length $\theta<\mathfrak{dp}(U)$ is $\supseteq^*$-bounded.
\end{proof}
\begin{lemma}\label{Lemma: making decreasing}
For any sequence $\l X_i\mid i<\theta\r\subseteq U$ for $\theta\leq\mathfrak{dp}(U)$, there is $\l X_i^*\mid i<\theta\r\subseteq U$ such that $X_i^*\subseteq X_i$ for all $i$ which is $\subseteq^*$-decreasing. 
\end{lemma}
\begin{proof}
    We construct $X_i$ by induction. At successor step, we let $X^*_{i+1}=X_{i+1}\cap X^*_i$. At limit steps $\alpha$, since $\alpha<\mathfrak{dp}(U)$, the sequence $\l X_i^*\mid i<\alpha\r$ which by induction is $\subseteq^*$-decreasing has a $\subseteq^*$-lower bound $A\in U$. We let $X^*_\alpha=A\cap X_\alpha$.
\end{proof}

Recall that an ultrafilter $U$ over $\kappa\geq\omega$ is called a $P_\lambda$-point, if $(U,\supseteq^*)$ is $\lambda$-directs. A $p$-point is a $P_{\kappa^+}$-point.  \begin{corollary}\label{Cor: Plambda}
    $\mathfrak{dp}(U)$ is the unique $\lambda$ such that $U$ is $P_\lambda$-point but not a $P_{\lambda^+}$-point.
\end{corollary} Most of the following propositions regarding $Sp_T(U)$, can be derived from the general set up:
\begin{proposition}\label{prop: propertied of dp and ch}
    Let $U$ be a uniform ultrafilter, then:
    \begin{enumerate}
        \item $\mathfrak{dp}(U)\leq cf(\mathfrak{ch}(U))\in Sp_T(U)$.
        \item $\min(Sp_T(U))=crit(j_U), \min(Sp_T(U,\supseteq^*))=\mathfrak{dp}(U)$.
        \item $\mathfrak{ch}(U)$ is an upper bound for $Sp_T(U)$. 
    \end{enumerate}
\end{proposition}
\begin{proof}
For $(1)$, we have already seen that $cf(\mathfrak{ch}(U))\in Sp_T(U)$ by Theorem \ref{thm: second kanamori improvment}. Suppose otherwise that $cf(\mathfrak{ch}(U))<\mathfrak{dp}(U)$ and fix $\l \theta_j\mid j<cf(\mathfrak{ch}(U))\r$ cofinal in $\mathfrak{ch}(U)$. Let $\l b_i\mid i<\mathfrak{ch}(U)\r$ be a base for $U$. For each $\theta_j$, by minimality, there is $x_j\in U$ which is not $\subset^*$-generated by $\l b_i\mid i<\theta_j\r$. Since  $cf(\mathfrak{ch}(U))<\mathfrak{dp}(U)=l(U,\supseteq^*)$, the sequence $\l x_j\mid j<cf(\mathfrak{ch}(U))\r$ has a pseudo-intersection  $x^*\in U$. Then $x^*$ cannot be $\subseteq^*$-generated by any base element, contradiction.

For $(2)$, if $\theta<crit(j_U)$, or $\theta<\mathfrak{dp}(U)=l(U,\supseteq^*)$, then $\theta\notin Sp_T(U)$ or $\theta\notin Sp_T(U,\supseteq^*)$ respectively, as any sequence of length $\theta$ is bounded. To see for example that $\mathfrak{dp}(U)\in Sp_T(U,\supseteq^*)$ (the proof that $crit(j_U)\in Sp_T(U)$ is completely analogous), we note that $\mathfrak{dp}(U)\in Sp_{dp}(U)\subseteq Sp_T(U)$. To see $(3)$, let  $\mathfrak{ch}(U)<\lambda$ be regular. Then given any $\lambda$-many sets in $U$, $\lambda$-many of them must contain the same element from a fixed base of size $\mathfrak{ch}(U)$. Hence $U$ will be $(\lambda,\lambda)$-cohesive, namely $\lambda\notin Sp_T(U)$.
\end{proof}
\begin{corollary}\label{cor minimum spec}
    If $\mathfrak{dp}(U)\neq \kappa$, then $\min(Sp_T(U)\setminus\{\kappa\})=\mathfrak{dp}(U)$.
\end{corollary}
\begin{corollary}
    If $\mathfrak{ch}(U)$ is regular, then $\mathfrak{ch}(U)=\max(Sp_T(U))$.
\end{corollary}
Note that using Theorem \ref{thm:Characterization of Galvin}, we can characterize in terms of the ultrapower $\sup Sp_T(U)$ as the least $\mu$ such that for every $\lambda>\mu$ regular, $M_U$ does not have a thin cover for $j_U''\lambda$. Hence the previous corollary provides an ultrapower characterization of $\mathfrak{ch}(U)$, whenever this cardinal is regular.  
    Ultrafilters with a singular character exists, for example, if we add $\kappa^{+\kappa^+}$-many Cohen function to $\kappa$ in a model of $GCH$, then in the extension the $2^\kappa=\kappa^{+\kappa^+}=\mathfrak{ch}(U)$ for all $U$ on $\kappa$. However, in the model above we will still have that $\mathfrak{ch}(U)=\sup(Sp_T(U))$ and therefore the proposed ultrapower characterization remains valid. Hence the following question is natural:
    \begin{question}
        Is it true that $\mathfrak{ch}(U)=\sup(Sp_T(U))$ for any uniform ultrafilter $U$? how about $\kappa$-complete ultrafilters over $\kappa$?
    \end{question}
    In that direction Isbell proved \cite{Isbell65} then if every singular cardinal is strong limit then for every $\mathbb{P}$, $u(\mathbb{P})=\sup(Sp_T(U))$. However, this theorem is irrelevant for us, since if $\mathfrak{ch}(U)$ is singular then $\kappa<\mathfrak{ch}(U)\leq 2^\kappa$, and Isbell's theorem does not apply.

    One my wonder what is the possible cofinalities of $\mathfrak{ch}(U)$. It is not hard to prove that for $\kappa$-complete ultrafilters $cf(\mathfrak{ch}(U))\geq\kappa^+$. If we give up $\kappa$-completeness we will see in Proposition \ref{prop: small cofinality character} that small cofinality is consistent.


\begin{proposition}\label{prop:product spectra}
    \begin{enumerate}
        \item For any two $\kappa$-complete ultrafilters over $\kappa>\omega$, $Sp_T(U\cdot W)=Sp_T(U)\cup Sp_T(W)$.
        \item For any two ultrafilter $U,W$ on $\omega$ such that $W\cdot W\equiv_TW$,  $Sp_T(U\cdot W)=Sp_T(U)\cup Sp_T(W)$.
    \end{enumerate} 
\end{proposition}
\begin{proof}
We prove $(1),(2)$ together, since both $U,W\leq_T U\cdot W$, $Sp_T(U\cdot W)\supseteq Sp_T(U)\cup Sp_T(W)$. In the other direction, if $\lambda\in Sp_T(U\cdot W)$, then $U\cdot W$ is not $(\lambda,\lambda)$-cohesive. By Theorem \ref{thm: improvement of product} for $(1)$, or Corollary \ref{cor: the case of omega} for $(2)$,     either $U$ or $W$ are not $(\lambda,\lambda)$-cohesive, namely, $\lambda\in Sp_T(U)\cup Sp_T(W)$.
\end{proof}
This cannot be improved, even just for $p$-point. Indeed,  it is consistent that there are $p$-point $W,U$ on $\omega$ such that $\mathfrak{ch}(W),\mathfrak{ch}(U)<\mathfrak{d}$ (see the remark following \cite[Cor. 36]{Dobrinen/Todorcevic11} and 
 other constructions where $U\not\geq_T\omega^\omega$), and $\mathfrak{d}$ is regular. By Milovich \cite{Milovich12},  $U\cdot W\equiv_T U\times W\times \omega^\omega$ and therefore $\mathfrak{d}\leq_T U\cdot W$. On the other hand,  $\mathfrak{d}\notin Sp_T(U)\cup Sp_T(W)$ as both $Sp_T(U)$ and $Sp_T(W)$ are bounded by $\max(\mathfrak{ch}(U),\mathfrak{ch}(W))<\mathfrak{d}$.

\subsection{Strong generating sequences and $P_\lambda$-points.}

\begin{definition}\label{definition: generating sequence}
    Let $U$ be an ultrafilter over $\kappa$. A sequence $\l A_\alpha\mid \alpha<\lambda\r$ such that:
    \begin{enumerate}
        \item If $\alpha<\beta<\lambda$ then $A_\beta\subseteq^* A_\alpha$.
        \item For every $X\in U$ there is $\alpha<\lambda$ such that $A_\alpha\subseteq^* X$.
    \end{enumerate}
    is called a \textit{strong generating sequence} for $U$.
\end{definition}
\begin{remark}
    For example, if $U$ is a $p$-point and $2^\kappa=\kappa^+$ then $U$ has a strong generating sequence of length $\kappa^+$.
\end{remark}
 Garti and Shelah (see \cite{BrookeTaylor2015CardinalCA} for details) noticed that
    assume $\kappa$ is supercompact and $\lambda>\kappa$ is regular, it is consistent that there is a $\kappa$-complete ultrafilter $U$ over $\kappa$ with a strong generating sequence of length $\lambda$.

\begin{lemma}
    If $U$ has a generating sequence of length $\lambda$, then $\lambda>\kappa$, and $U$ is a $P_\lambda$ point.
    Moreover, $\mathfrak{dp}(U)= \lambda$.
\end{lemma}
\begin{proof}
    Let $\l A_\alpha\mid \alpha<\lambda\r$ be a strong generating sequence. Given any sequence $\l X_\alpha\mid \alpha<\theta\r$ where $\theta<\lambda$, for every $\alpha<\theta$ there is $\beta_\alpha<\lambda$ such that $A_{\beta_\alpha}\subseteq^* X_\alpha$. Take $\beta^*=\sup_{\alpha<\theta}\beta_\alpha$, then $A_{\beta^*}$ $\supseteq^*$-bounds the $X_\alpha$'s. For the ``Moreover" part, by Corollary \ref{Cor: Plambda} $\lambda\leq \mathfrak{dp}(U)$. As a generating sequence, $\l A_\alpha\mid \alpha<\lambda\r$ cannot have a pseudo intersection. Hence $\lambda\in Sp_{dp}(U)$, and the equality holds.
\end{proof}

\begin{corollary}\label{Cor: Spec for strong generating}
    If $U$ has a strong generating sequence of length $\lambda$, then $\mathfrak{dp}(U)=\mathfrak{ch}(U)$ and  $$Sp_{dp}(U)=Sp_{T}(U)\setminus\{\kappa\}=\{\lambda\}.$$
\end{corollary}
\begin{proof}
    By Corollary \ref{cor minimum spec} and Proposition \ref{prop: propertied of dp and ch}, $$\kappa<\lambda=\mathfrak{dp}(U)=\min(Sp_T(U)\setminus\{\kappa\})\leq\sup(Sp_T(U))\leq \mathfrak{ch}(U)\leq \lambda.$$  Hence $$\{\lambda\}\subseteq Sp_{dp}(U)\subseteq Sp_{T}(U)\setminus \{\kappa\}=\{\lambda\}.$$
\end{proof}
The point spectrum  characterizes the existence of strong generating sequences:
\begin{lemma}\label{Lemma: classifying strong generate}
    Let $U$ be a uniform ultrafilter on $\kappa$. $U$ has a strongly generating sequence of length $\lambda$ if and only if $\mathfrak{dp}(U)=\mathfrak{ch}(U)=\lambda$. 
\end{lemma}
\begin{proof}
    One direction follows from Corollary \ref{Cor: Spec for strong generating}. In the other direction, suppose that $\mathfrak{dp}(U)=\mathfrak{ch}(U)=\lambda$, given a base $\l X_i\mid i<\lambda\r$ for $U$, we may apply Lemma \ref{Lemma: making decreasing} to the base to obtain a strong generating sequence for $U$ of length $\lambda$.
\end{proof}

\begin{corollary}
    If $\kappa$ is supercompact cardinal, it is consistent that $\kappa$ carries a $\kappa$-complete ultrafilter which is $(\kappa^+,\kappa^+)$-cohesive. 
\end{corollary}
In the next section we address the question of consistency strength of such ultrafilters.
\section{Consistency results}\label{Section: Consistency}

Although the existence of a long generating sequence seems stronger than Kanamori's question, it is actually equivalent:
\begin{theorem}
    For any $\lambda>\kappa$ The following are equicosistent:
    \begin{enumerate}
        \item there exists of an ultrafilter with a strong generating sequence of length $\lambda$.
        \item there is a $\kappa$-complete ultrafilter $U$ with $\min(Sp_T(U)\setminus\{\kappa\})=\lambda$; $U$ is $(\mu,\mu)$-cohesive for every regular $\mu\in (\kappa,\lambda)$.
        \item there exists of a $\kappa$-complete ultrafilter $U$ with $\mathfrak{dp}(U)=\lambda$.
\item There exists a $P_{\lambda^+}$-point over $\kappa$, and $\lambda$ is regular.
    \end{enumerate} 
\end{theorem}
\begin{proof}
    $(1)$ implies $(2)$  by Corollary \ref{Cor: Spec for strong generating}. To see that $(2)$ implies $(3)$, let $U$ witness $(2)$, and let $U_0\leq_{RK} U$ be normal. Then $Sp^*_T(U_0)\subseteq Sp^*_T(U)$, and  $\mathfrak{dp}(U_0)>\kappa$. Hence $\mathfrak{dp}(U_0)\geq \min(Sp_T(U)\setminus\{\kappa\})=\lambda$. To see that $(3)$ implies $(4)$, by \ref{Cor: Plambda}, $U$ is a $P_{\lambda^+}$-point and we already made the observation that $\mathfrak{dp}(U)$ must be regular. Finally for $(4)$ implies $(1)$, 
    suppose that $U$ is a $P_{\lambda^+}$-point for some regular $\lambda>\kappa$, then $U$ a $\kappa$-complete ultrafilter.  Levi collapsing $\mathfrak{ch}(U)$ to be $\lambda$ (if needed) is a $\lambda$-closed forcing, and does not introduce new $<\lambda$-sequence. Hence $U$ is still a $\kappa$-complete ultrafilter in the extension. Moreover, we have $\mathfrak{dp}^{V[G]}(U)=\lambda=\mathfrak{ch}^{V[G]}(U)$. Hence by Lemma \ref{Lemma: classifying strong generate}, $U$ had a strong generating sequence in $V[G]$.
\end{proof}
In particular, the existence of a $\kappa$-complete $(\kappa^+,\kappa^+)$-cohesive ultrafilter is equiconsistent with the existence of an ultrafilter having a strong generating sequence of length $\kappa^{++}$- a principle which we conjecture having very high  consistency strength. Let us us the following result of Gitik to conclude that the consistency strength of this principle is more the $o(\kappa)=\kappa^{++}$
The following is due to Gitik:
\begin{theorem}
    The consistency strength of the existence of a $P_{\kappa^{++}}$-point is greater than $o(\kappa)=\kappa^{++}+1$.
\end{theorem}
\begin{proof}
    Assume that the core model $\mathcal{K}$ exists. Let us start by showing that in $\mathcal{K}$, $o(\kappa)>\kappa^{++}$. Just otherwise, let $U$ be a $P_{\kappa^{++}}$-point in $V$ and let $j=j_U:V\rightarrow M$ be the ultrapower embedding. By Mitchel and Schindler \cite{MitCoherent,schindler_2006}, $j\restriction \mathcal{K}:\mathcal{K}\rightarrow \mathcal{K}^M$ is an iterated ultrapower of the measures/extenders of $\mathcal{K}$. Suppose that $j=i_{0,\theta}$ where $\l i_{\alpha,\beta}\mid 0\leq \alpha\leq \beta\leq \theta\r$ is a normal iteration such that $i_{\alpha,\alpha+1}$ is an ultrapower embedding by an extender/measure on some increasing sequence $\l \delta_\alpha\mid \alpha<\theta\r$. In particular $\delta_0=\kappa$. If $o^{\mathcal{K}}(\kappa)\leq\kappa^{++}$, then $i_{0,\theta}$ consists only of ultrapowers by measures. Let $U_0$ be the first normal ultrafilter on $\kappa$ used in the iteration, so in particular, $i_{0,1}$ is the ultrapower by $U_0$. We note that $U_0\subseteq U$, this follows since the critical point of $i_{1,\theta}$ is greater than $\kappa$. Since in $\mathcal{K}$, $2^\kappa=\kappa^{+}$, and since $U$ is a $P_{\kappa^{++}}$-point, we can find a single set $A\in U$ such that for every $B\in U_0$, $A\subseteq^* B$. But then $A\in M$ (since $M$ is close under $\kappa$-sequences) and therefore $U_0\in M$. Again, since the critical point of $i_{1,\theta}$ is greater than $\kappa$, $U_0$ is a $\mathcal{K}^M$-ultrafilter and by maximality of the core model $U_0\in\mathcal{K}^M$. This is impossible as it would mean that $U_0\in \mathcal{K}_{U_0}$ (the ultrapower of $\mathcal{K}$ by $U_0$), contradiction. 

    Now suppose $o(\kappa)=\kappa^{++}+1$, then in $\mathcal{K}$ we would have a coherent sequence of of length $\kappa^{++}+1$; that is, a coherent sequence over $\kappa$ consists of $\kappa^{++}$ measures denoted by $F(\kappa, \alpha)$, for  $\alpha< \kappa^{++}$ and a single $(\kappa,\kappa^{++})$- total extender $F(\kappa,\kappa^{++})$ which can now be used in the iteration $j_U\restriction\mathcal{K}$. Denote $E_0=F(\kappa,\kappa^{++})$ Note that if the iteration starts with a measure, we again reach a contradiction as above. So suppose that $j\restriction\mathcal{K}$ is a normal iteration that starts with $j_{E_0}$. First, we claim that $E_0$ and its images can only be used finitely many times at stages $\{\nu_n\mid n<\omega\}$. Since otherwise, letting $\nu_\omega=\sup_{n<\omega}\nu_n$ by \cite{GITIK1993227} proposition 1.3 and the proceeding remark, $i_{0,\nu_\omega}(E_0)$ is going to be a member of $\mathcal{K}^{M_U}$. But then at stage $\nu_\omega$, we must apply either a measure or an extender on $i_{0,\nu}(\kappa)$ (as otherwise, $i_{0,\nu}(\kappa)=j(\kappa)$ will have cofinality $\omega$, which is impossible). This is impossible, since   then we will be able to conclude that $i_{0,\nu}(E_0)$ is a member of its own ultrapower (or even worst, to the ultrapower by a measure which is Mitchell below it), this is a contradiction.  
    
     Let $E_0(\kappa)$ be the normal measure of $E_0$. Since the normal iteration $j\restriction\mathcal{K}$ starts with $E_0$, we have that $E_0(\kappa)=U\cap \mathcal{K}$. Again, by the same reasoning as above, we can find $A\in U$ such that for every $B\in E_0(\kappa)$, $A\subseteq^* B$. In particular, $j(A)\setminus\kappa\subseteq \bigcap j''E_0(\kappa)$. 
     It is not hard to see that every $\eta\in j(A)\setminus\kappa$ is a seed for $E_0(\kappa)$, namely, $$E_0(\kappa)=\{X\subseteq\kappa\mid \eta\in j(A)\}$$
     Suppose that $i=i_{0,\alpha+1}$ is the part of the iteration such that after $i$ we only iterate ultrapowers by measures. Since $j(A)\in j(U)$, then $j(A)\setminus\kappa$ is unbounded in $j(\kappa)$ and therefore there is $\eta\in j(A)$ above $i_{0,\alpha}(\kappa)<j(\kappa)$. We claim that $\eta\geq i(\kappa)$ and therefore $i(\kappa)<j(\kappa)$. Otherwise,   
      there is $a\in [i_{0,\alpha}(\kappa^{++})]^{<\omega}$, and a function $f:[\kappa]^{|a|}\rightarrow\kappa$ such that $i(f)(a)=\eta$. So the function $g:\kappa\rightarrow\kappa$ defined by $g(x)=\sup_{\vec{\xi}\in [x^{++}]^{|a|}}f(\vec{\xi})+1$ satisfies that $i(g)(i_{0,\alpha}(\kappa))> \eta$. Let $C_g$ be the club of closure points of $g$, then $C_g\in E_0(\kappa)$ and therefore $\eta\in i(C_g)$, which means that $i(g)(i_{0,\alpha}(\kappa))<\eta,$
    contradiction. By the same reasoning, we conclude that $\eta$ must be the image of $\kappa$ in the iteration. Namely there is some $\xi$ such that $i_{0,\xi}(\kappa)=\eta$. Since $\eta<j(\kappa)$, this means that at some stage $\xi\leq\xi'$, $i_{\xi',\xi'+1}$ is the ultrapower by some normal measure on $\eta$. Since $\eta$ is a seed for $E_0(\kappa)$, it must be that $i_{\xi',\xi'+1}$ is the ultrapower by $i_{0,\xi'}(E_0(\kappa))$. But then $i_{0,\xi'}(E_0(\kappa))\notin \mathcal{K}^{M_U}$. Thus we will reach the desired contradiction once we prove that $i_{0,\xi'}(E_0(\kappa))\in M_U$. 
    \begin{claim}
        $i_{0,\xi'}(E_0(\kappa))$ is generated by $\{j(B)\cap(\rho,\eta)\mid B\in E_0(\kappa), \ \rho<\eta\}$.
    \end{claim}
    \begin{proof}
        Let $X\in i_{0,\xi'}(E_0(\kappa))$, then $X=i_{0,\xi'}(f)(\lambda_1,...\lambda_n)$ for some $\lambda_1,...,\lambda_n<i_{0,\xi'}(\kappa)$ and $f:[\kappa]^n\to E_0(\kappa)$. Let $B=\Delta_{\vec{\xi}\in[\kappa]^n}f(\vec{\xi})$, then $i_{0,\xi'}(B)\setminus \lambda_n+1\subseteq X$. Finally, note that  the critical point of $i_{\xi',\infty}$ (the tail of the iteration) is $\eta$, hence $$j(B)\cap \eta=i_{\xi',\infty}(i_{0,\xi'}(B))\cap \eta=i_{0,\xi'}(B).$$
        We conclude that $j(B)\cap (\lambda,\eta)\subseteq X$.
    \end{proof}
    To finish the proof, we have that for each $B\in E_0(\kappa)$ and each $\kappa\leq\xi<\eta$, $j(A)\cap (\xi,\eta)\subseteq j(B)\cap (\xi,\eta)$. Hence $i_{0,\xi'}(E_0)=\{X\in \mathcal{K}^{M_U}\mid \exists \xi<\eta, j(A)\cap (\xi,\eta)\subseteq X\}$ is definable in $M_U$, contradiction. 
\end{proof}
A generalization of the above argument gives the following
\begin{corollary}
    If there is a $P_{\kappa^{++}}$-point and the core model exists, then $o^{\mathcal{K}}(\kappa)$ is at least $ \kappa^{++}+\kappa$
\end{corollary}
\begin{proof}

Assume otherwise, there there are $\l E_i\mid i<\bar{\kappa}\r$ for $\bar{\kappa}<\kappa$ is the extenders in the coherent sequence on $\kappa$. In particular for every step of the iteration $\alpha$, $i_{0,\alpha}(\l E_i\mid i<\bar{\kappa}\r)=\l i_{0,\alpha}(E_i)\mid i<\bar{\kappa}\r$ includes all the extenders on $i_{0,\alpha}(\kappa)$.
Let 
$C_0$ enumerate only the steps of the iteration where we performed an extender ultrapower over an images of $\kappa$. Note that $C_0$ is closed below $j(\kappa)$ as the iteration is normal and by \cite{GITIK1993227}. If $C_0$ would have been bounded, then there is a last extender that is being performed, and the argument of the previous argument applies. For each $\nu\in C_0$ there is an index $\alpha_\nu< \bar{\kappa}$ such that $i_{\nu,\nu+1}$ is an ultrapower extender by $F(\nu,\alpha_\nu)$. Since $j(\kappa)$ has $V$-cofinality greater than $\kappa$, there is an $\omega+1$ copy inside $C_0$ where we applies the same index of extenders. This is a contradiction. 

\end{proof}
\begin{corollary}
    If there is a $\kappa$-complete, $(\kappa^+,\kappa^+)$-cohesive ultrafilter on $\kappa$ then there is an inner model with a measurable cardinal $\lambda$ such that $o(\lambda)\geq\lambda^{++}+\lambda$.
\end{corollary}
This can be pushed even higher, but we leave the following problem:
\begin{question}
    What is the consistency strength of a $P_{\lambda}$-point for $\lambda>\kappa^+$?
\end{question}
Equivalently, 
\begin{question}
    What is the consistency strength of the existence of a $\kappa$-complete $(\kappa^+,\kappa^+)$-cohesive ultrafilter over $\kappa$?
\end{question}

Next we would like to treat the possible complexity of $Sp_T(U)$. First, we would like to rise the following question:
\begin{question}
    Is it consistent that $Sp_T(U)$ is not an interval of regular cardinals?
\end{question}
A natural approach would be to take two ultrafilter $U,W$ such that $U$ has a strong generating sequence of length $\kappa^+$ (for example) and $W$ has a strong generating sequence of length $\kappa^{+++}$. By proposition \ref{prop:product spectra} $Sp_T(U\cdot W)=Sp_T(U)\cup Sp_T(W)=\{\kappa^+,\kappa^{+++}\}$. We do not know if it is consistent to have two such ultrafilters. However, the usual way to obtain such ultrafilters; that is, to iterate Mathias forcing $\lambda$-many times is doomed. Indeed if we iterate the above forcing to produce one ultrafilter with a strong generating sequence of length $\lambda$, and then iterate again $\lambda'$-many times to generate an ultrafilter with a strong generating sequence of length $\lambda'$. If $\lambda,\lambda'$ have different cofinalities, any ultrafilter in the generic extension strongly generated by $cf(\lambda)$-many sets
would have to be generated from an intermediate step of the iteration. The following lemma assures that adjoining even one Mathias generic set prevents ultrafilters from being generated by the ground model.
\begin{lemma}
    Let $a_G$ be an $\mathbb{M}_U$-generic set. Denote by $$R=\bigcup_{i<\kappa}[(a_G)_{2i},(a_G)_{2i+1}).$$ Then both $R$ and $\kappa\setminus R$ cannot contain an unbounded set from the ground model. 
\end{lemma}
\begin{proof}
    Suppose otherwise that $X\in V$ and $X\subseteq R$ (the proof for $\kappa\setminus R$ is similar). Let $\l a,A\r\in G$ be a condition such that $\l a,A\r\Vdash_{\mathbb{M}_U} \check{X}\subseteq \lusim{R}$. We may assume that $\max(a)=a_{2i+1}$ for some $i<\kappa$ (namely $otp(a)$ is an even successor ordinal). Since $X$ is unbounded, there is $x\in X\setminus a_{2i+1}+1$, and since $A\in U$, there is $\alpha\in A\setminus X$. extend $\l a,A\r$ to $\l a\cup\{\alpha\},A\setminus \alpha+1\r$. This condition forces that $x\in X\setminus R$, contradiction. 
\end{proof}

\begin{lemma}\label{Lemma: Cohens blowup Tukey}
     Let $V[G]$ be the usual Cohen extension where we added $\kappa^{++}$-many Cohen generic functions from $\kappa$ to $2$ to a model of $2^\kappa=\kappa^+$. Then any uniform ultrafilter $U\in V[G]$ is not $(\kappa^{++},\kappa^+)$-cohesive. In particular it is not $(\kappa^+,\kappa^+)$-cohesive and not $(\kappa^{++},\kappa^{++})$. Namely $ Sp_{T}(U)\supseteq\{\kappa,\kappa^+,\kappa^{++}\}$.
\end{lemma}
\begin{proof}
    Let $G$ be $V$-generic for $Add(\kappa,\kappa^{++})$. Let $U\in V[G]$ be an ultrafilter and let $\l X_\alpha\mid \alpha<\kappa^{++}\r$ be the Cohen mutually generic subsets of $\kappa$ added by $G$. Note that for every $\alpha$, either $X_\alpha\in U$ or $X_\alpha^c\in U$. Since the complement of a Cohen generic set is also Cohen generic, we may assume that $\l X_\alpha\mid \alpha<\kappa^{++}\r\subseteq U$. Let us prove that the sequence of Cohen generics witnesses that $U$ is not $(\kappa^{++},\kappa^+)$-cohesive. Suppose towards a contradiction that for some $I$ of size $\kappa^{+}$, $Y=\bigcap\{X_i\mid i\in I\}\in U$. In particular $Y$ is unbounded in $\kappa$ (since $U$ is uniform). Since $Y\subseteq\kappa$, by $\kappa^+$-c.c, there is $J\subseteq\kappa^{++}$ such that $|J|=\kappa$ and $Y\in V[G\restriction J]$. Pick any $\beta\in I\setminus J$ which exists since $|I|>|J|$. By mutual genericity, $X_i$ is generic over $V[G\restriction J]$, in particular $X_i$ cannot contain a set of size $\kappa$ from $V[G\restriction J]$, contradicting $Y\subseteq X_i$. 
\end{proof}
By Woodin and Gitik, starting with a model of GCH and a measurable $\kappa$  with Mitchell order $o(\kappa)=\kappa^{++}$, it is possible to get a generic extension in which $\kappa$ is measurable and $\kappa^{++}$-many mutually generic Cohen functions (over some intermediate model where $2^\kappa=\kappa^+$).
\begin{corollary}
    In the model above, for every ultrafilter $U$, $\mathfrak{dp}(U)<\mathfrak{ch}(U)$.
\end{corollary}
Clearly, the previous argument generalized for  $\lambda>\kappa^{++}$ to produce a model where $2^\kappa=\lambda$ and for every uniform ultrafilter $U$ over $\kappa$, $U$ is not $(\lambda,\kappa^+)$-cohesive (so also all the instance of $(\rho,\mu)$-cohesiveness fail for $\kappa^+\leq\mu\leq\rho\leq\lambda$). This improves Theorem 4.2 from \cite{bgp}, which proves this consistency (with the same model!) only for normal ultrafilters.
 The argument from the previous lemma can be used to give an example of an ultrafilter with a character of small cofinality:
\begin{proposition}\label{prop: small cofinality character}
        Relative to a measurable cardinal, it is consistent that there is a uniform ultrafilter on a regular cardinal $\kappa$ with $cf(\mathfrak{ch}(U))=\omega_1$
    \end{proposition}
    \begin{proof}
        Raghavan and Shelah \cite{RagShel} proved that after adding $\kappa^{+\omega_1}$-many Cohen reals, we can find a uniform ultrafilter $D$ in the extension with $\mathfrak{ch}(D)\leq \kappa^{+\omega_1}$. 
        Let us prove that in fact $\mathfrak{ch}(D)=\kappa^{+\omega_1}$. We think of the forcing $Add(\kappa^{+\omega_1},\omega)$ as adding $\kappa^{+\omega_1}$-many characteristic sets to $\kappa$ $\l X_\alpha\mid \alpha<\kappa^{+\omega_1}\r$ with finite approximation.
        The ultrafilter $D$ have to pick for each $\alpha<\kappa^{+\omega_1}$ either $X_\alpha\in D$ or $\kappa\setminus X_\alpha\in D$. As in the proof of \ref{Lemma: Cohens blowup Tukey}, we may assume that for every $\alpha<\kappa^{+\omega_1}$, $X_\alpha\in D$. If $\mathfrak{ch}(D)<\kappa^{+\omega_1}$, we could have found $\omega_1$-many sets $X_\alpha$ for $\alpha\in I$ such that $\bigcap_{\alpha\in I}X_\alpha\in D$.
        Let $Y$ be  any countable subset of $\bigcap_{\alpha\in I}X_\alpha$. Using the chain condition, we see that $Y$ belongs the extension by countably many of the Cohen reals, and therefore to the extension by countably many of the $X_\alpha$'s. That is $Y\in V[\l X_\alpha\mid \alpha\in J\r]$ where $J\in[ I]^{\omega}$. Pick any $\alpha^*\in I\setminus J$. Then $X_\alpha$ is generic over $V[\l X_\alpha\mid \alpha\in J\r]$.  However, by genericity,  the set $X_{\alpha^*}$ is only finitely approximated in $V[\l X_\alpha\mid \alpha\in J\r]$  which produces a contradiction.
    \end{proof}
    \begin{question}
         Is there a strong limit regular cardinal $\kappa$ carrying a uniform ultrafilter $U$ such that $cf(\mathfrak{ch}(U))<\kappa$?
    \end{question}
What about $Sp_{dp}(U)$ in the Cohen extension? Alan Dow informed us that Kunen proved in his master thesis that not only there are no towers of length $\aleph_2$, but there are no $\subseteq^*$ chains of ordertype $\aleph_2$ in the extension by the finite support product of $\aleph_2$-many reals. He also suggested a simpler argument for towers which generalizes to regular uncountable cardinals. We need a slight strengthening of that argument to show that Cohen forcing does not add chains modulo bounded of length $\kappa^{++}$ to a normal ultrafilter $U$ which are unbounded in $(U,\supseteq^*)$. 
\begin{lemma}
Suppose $2^\kappa=\kappa^+$. After forcing with $Add(\kappa,\kappa^{++})$, there is no normal ultrafilter with a sequence $\l A_i\mid i<\kappa^{++}\r\subseteq U$ which is $\subseteq^*$-decreasing, and unbounded in $(U,\subseteq^*)$. Equivalently, $\kappa^{++}\notin Sp_{dp}(U)$.
\end{lemma}
\begin{proof}
    Suppose otherwise and let $\dot{\mathcal{U}},\dot{\mathcal{A}}$ be names and $p\in Add(\kappa,\kappa^{++})$ forcing all the relevant information. We may assume that $\mathcal{A}=\l \dot{A}_i\mid i<\kappa^{++}\r$ and $\mathcal{U}=\l \dot{X}_i\mid i<\kappa^{++}\r$ are sequences of nice names, and $p$ forces that $\mathcal{U}$ is an ultrafilter and $\mathcal{A}$ is $\subseteq^*$-decreasing and unbounded in $(\mathcal{U},\subseteq^*)$. Let $M\prec H_\theta$ be an elementary submodel of size $\kappa^+$, closed under $\kappa$-sequence with $$p,\l \dot{X}_i\mid i<\kappa^{++}\r,\l\dot{A}_i\mid i<\kappa^{++}\r...\in M.$$  Also assume that $M\cap \kappa^{++}=\delta\in\kappa^{++}$. First we note that for each $i<\delta$, $\dot{X}_i,\dot{A}_i\in M$ is a nice name, and since $M$ is closed under $\kappa$-sequnces and by $\kappa^{+}$-cc, $\dot{X}_i$ and $\dot{A}_i$ are in fact names of $Add(\kappa,\delta)$. 
    \begin{claim}
        In $V[G\restriction \delta]$ we have:
        \begin{enumerate}
            \item $\{ (\dot{X}_i)_{G\restriction\delta}\mid i<\delta\}$ is the ultrafilter $\mathcal{U}_\delta=(\dot{\mathcal{U}})_G\cap V[G\restriction\delta]$
            \item $\mathcal{U}_\delta$ is normal.
            \item $\l (\dot{A}_i)_{G\restriction\delta}\mid i<\kappa^{++}\r$ is a tower in $\mathcal{U}_\delta$. \end{enumerate}
    \end{claim} 
\end{proof}
\begin{proof}[Proof of Claim]
Clearly, $\{(\dot{X}_i)_{G\restriction\delta}\mid i<\delta\}\subseteq\mathcal{U}_\delta$. Hence it suffices to prove that $\{(\dot{X}_i)_{G\restriction\delta}\mid i<\delta\}$ is an ultrafilter in $V[G\restriction\delta]$. Suppose not, and let $\dot{X}$ be a nice name such that for some $q\geq p$, $q\in G\restriction\delta$, $$q\Vdash_{Add(\kappa,\delta)} \dot{X},\kappa\setminus\dot{X}\notin\{\dot{X}_i\mid i<\delta\}.$$ Again, by closure of $M$ to $\kappa$-sequences and the $\kappa^+$-cc, $\dot{X},q\in M$. Hence for every $i<\delta$, $M\models q\Vdash \dot{X}\neq \dot{X}_i\wedge \kappa\setminus\dot{X}\neq\dot{X}_i.$ Since $\kappa^{++}\cap M=\delta$, $$M\models q\Vdash`` \dot{\mathcal{U}}\text{ is not an ultrafilter}",$$ contradiction. Normality just follows from the $\kappa^+$-cc and the fact that cofinality of $\delta$ is $\kappa^+$.  A similar argument shows that $\l (\dot{A}_i)_G\mid i<\kappa^{++}\r$ is a tower in $\mathcal{U}_\delta$.
\end{proof}
Note that $\dot{A}_{\delta}$ is forced by $p$ to be $\subseteq^*$-bound, so by $\kappa^+$-cc, over the model $V[G\restriction\delta]$ we can find $\kappa$-many coordinates $I\subseteq[\delta,\kappa^{++})$ such that in $V[G\restriction\delta][G\restriction I]$ we will have a $\subseteq^*$-bound. Note that $G\restriction I$ is forcing equivalent to the extension by a single Cohen $Add(\kappa,1)$. Hence it suffices to prove the following claim:
\begin{claim}
    Suppose that $\mathcal{W}$ is a normal ultrafilter ($p$-point is enough) on $\kappa$, $\l A_i\mid i<\kappa^+\r$ is a tower in $\mathcal{W}$. Then adding a single Cohen function to $\kappa$ does not add pseudo intersection $X$ for $\l A_i\mid i<\kappa^{+}\r$ which is positive with respect to the filter generated by $\mathcal{W}$ in the extension.
\end{claim}
\begin{proof}[Proof of claim.] Suppose otherwise, the $\dot{A}$ is a name and $p$ a condition such that $$p\Vdash`` \dot{A}\in(\overline{\mathcal{W}})^+\text{ is a pseudo intersection of }\l A_i\mid i<\kappa^+\r",$$
where $\overline{\mathcal{W}}$ is the filter generated by $\mathcal{W}$ in the extension. For each $q\geq p$ let $A_q=\{\alpha<\kappa\mid \exists r\geq q, \ r\Vdash \alpha\in \dot{A}\}$. Then $A_q$ is in the ground model cover of $\dot{A}$, hence positive for $\mathcal{W}$. Therefore $A_q\in \mathcal{W}$. Since $Add(\kappa,1)$ has size $\kappa$, and $\mathcal{W}$ is assumed to be a $p$-point, there is $A^*\in \mathcal{W}$ which is a pseudo intersection for all the $A_q$'s. Since $A_i$ is a tower, there is $i^*<\kappa^{+}$ such that $A^*\setminus A_i$ is unbounded in $\kappa$. Find $p_0\geq p$ such that $p_0\Vdash \dot{A}\setminus \xi\subseteq A_i$. Find $\xi'$ such that $A^*\setminus \xi'\subseteq A_{p_0}$ and find $\rho>\xi,\xi'$ such that $\rho\in A^*\setminus A_i$. Hence $\rho \in A_{p_0}$ and by definition there is $p'\geq p_0$ forcing that $\rho\in \dot{A}$. But $\rho\notin A_i$ and $p'$ is suppose to force also that $\dot{A}\setminus \xi\subseteq A_i$, contradiction. 
\end{proof}

In order to add values of the depth spectra, let us prove the following sufficient condition. We will use the following notations:
$$Cl(X)=\{\alpha<\sup(X)\mid \sup(X\cap\alpha)=\alpha\}\cup X$$
\begin{lemma}
    Suppose that $M$ is a ZFC transitive model $\kappa<\lambda$ regular cardinals of $M$ and $\l M_i\mid i<\lambda\r$ , $\l\mathcal{U}_i\mid i<\lambda\r$, $\l k_i\mid i<\lambda\r$ are such that: 
    \begin{enumerate}
        \item If $i<j$ then $M_i\subseteq M_j$.
        \item $\mathcal{U}_i$ is a normal $M_i$-ultrafilter.
        \item $k_i\subseteq \kappa$ diagonalizes $\mathcal{U}_i$ (i.e. $k_i\subseteq^* X$ for every $X\in \mathcal{U}_i$). Also $k_j\in M_{i}$ for all $j<i$. 
        \item $P(\kappa)^M=\bigcup_{i<\lambda}P(\kappa)^{M_i}$.
    \end{enumerate}
    Then $\l Cl(\kappa_i)\mid i<\lambda\r$ is a $\subseteq^*$-decreasing sequence of clubs which generates the club filter of $M$.  In particular,  $\l Cl(k_i)\mid i<\lambda\r$ is unbounded.
\end{lemma}
\begin{proof}
    To see that $Cl(k_i)$ are $\subseteq^*$-decreasing, let $j<i$, since $k_j\in M_{i}$, then $Cl(k_j)$ is definable in $M_{i}$ and since $\mathcal{U}_{i}$ is normal, $Cl(k_j)\in \mathcal{U}_{i}$. By condition $(3)$ it follows that $k_{i}\subseteq^* Cl(k_j)$. Hence $Cl(k_{i})\subseteq Cl(k_j)$. Next, let $C\in M$ be a club at $\kappa$, then by $(4)$, $C\in M_i$ for some $i$ and therefore by normality $C\in \mathcal{U}_i$. It follows as before that $Cl(k_{i})\subseteq^* C$. 
\end{proof}
To realize the above situation while keeping $\kappa$ measurable, we use the Extender-based Magidor-Radin forcing of Merimovich \cite{CarmiRadin}. We do not need to define the forcing here as the construction of the models together with all the required verification was done in \cite{BenNeriaGitik}, and we simply cite the relevant propositions and definitions from their paper. 
\begin{theorem}
    Let $\kappa$ be a measurable cardinal, $\vec{E}=\l E_\alpha\mid \alpha<\lambda\r$, $\lambda>\kappa^+$ is regular which is not the successor of a singular of cofinality at most $\kappa$. Suppose that $\vec{E}$ has a local repeat point\footnote{See \cite[Def. 5.9]{CarmiRadin}}.  Then in the generic extension by $\mathbb{P}_{\bar{E},\lambda}$, $\kappa$ is measurable and there is a $\subseteq^*$-decreasing sequence of clubs of length $\lambda$ which generates the club filter. 
\end{theorem}
\begin{proof}
    We follow the argument of \cite{BenNeriaGitik}. Let $G\subseteq \mathbb{P}_{\bar{E},\lambda}$ be $V$-generic and let $M=V[G]$. By Merimovich \cite{CarmiRadin}, in $V[G]$, $2^\kappa\leq\lambda$ and $\kappa$ is measurable. Let $\l \dot{X}_i\mid i<\lambda\r$ be a sequence of nice names which enumerates $P(\kappa)^{V[G]}$. As in \cite[Thm. 2.13]{BenNeriaGitik}, for each $\eta<\lambda$, we can find and elementary submodels $N_\eta\prec H_\theta$, such that $\{\dot{X}_i\mid i<\eta\}\subseteq N_\eta$ and  $N_\eta$ satisfies the assumptions describe at the top of \cite[Sec.~ 2.1]{BenNeriaGitik}. Moreover, it is possible to make $\l N_\eta\mid \eta<\lambda\r$ increasing with respect to inclusion. By  \cite[Lemma 2.1\& Prop. 2.4]{BenNeriaGitik}, for each $\eta<\lambda$, $\mathbb{P}_{\bar{E},\lambda}\cap N_\eta$ is a complete subforcing of $\mathbb{P}_{\bar{E},\lambda}$ which is isomorphic to $\mathbb{P}_{\bar{E}\restriction \delta_\eta,\delta_\eta}$, where $\delta_\eta=N_\eta\cap\lambda$. Letting $M_\eta=V[G\cap N_\eta]$, considition $(1)$ is satisfied as the $N_\eta$'s are increasing. By the $\kappa^{++}$-cc of the forcing, we may assume that each $\dot{X}_i$ for $i<\eta$ is a $\mathbb{P}_{\bar{E},\lambda}\cap N_\eta$-name and therefore $X_i\in M_\eta$ for every $i<\eta$. Hence $P(\kappa)^{V[G]}=\bigcup_{i<\lambda}P(\kappa)^{M_i}$ and condition $(4)$ is satisfied. 
    The existence of the ultrafilter $\mathcal{U}_\eta$ follows from \cite[Def. 2.6 \& Prop. 2.7]{BenNeriaGitik} and finally, the set $k_\eta$ exists by \cite[Def. 2.10 \& Prop. 2.12]{BenNeriaGitik}. It is evident from the definition of $k_\eta$ that it can be computed in $V[G\cap N_{\eta+1}]$
    (since we only need to compare $o(\nu(\bar{\kappa}))$ and $\nu(\bar{\delta}_\eta)_0$, which is possible as $\delta_\eta\in N_{\eta+1}$). However, we could simply require in the inductive construction of the sequence $N_{\eta}$ that also $\dot{k}_\rho\in N_\eta$ for $\rho<\eta$.
\end{proof}
To preserve measurability, the assumption that $\vec{E}$ has a local repeat point can be improved. For example, consider the case $\lambda=\kappa^{++}$, then we only need to ensure the existence of a local repeat point (see \cite[Def. 2.2]{BenNeriaGitik}) $\delta$ of cofinality $\kappa^{++}$ which is a limit of local repeat points. Since it is possible to produce such local repeat by elementary submodels $N_\eta$ as in the proof of the previous theorem, the existence of such $\delta$ is below $o(\kappa)=\kappa^{+3}$.

\begin{question}
    What is the consistency strength of having $\kappa^{++}\in Sp_{dp}(U)$ for a $\kappa$-complete ultrafilter $U$?
\end{question}
\subsection*{Acknowledgment} The author would like to thank Moti Gitik for insightful discussions which greatly contributed to the last section of this paper and to Alan Dow for pointing out to us several results on $\omega$ which inspired some of the consistency results of the last section. Also to Gabe Goldberg for countless hours of mathematical and spiritual support.
\bibliographystyle{amsplain}
\bibliography{ref}
\end{document}